\documentclass[11pt]{amsart}
\usepackage{mathrsfs,amsthm,amsmath,amssymb,bbm,pdfsync,prettyref,hyperref,graphicx,fullpage,esint}

\usepackage{wasysym}
\graphicspath{ {../plots/} }

\theoremstyle{plain}
\newtheorem{thm}{\protect\theoremname}[section]
\newtheorem*{thm*}{\protect\theoremname}

\newtheorem{lem}[thm]{\protect\lemmaname}
\newtheorem*{lem*}{\protect\lemmaname}

\newtheorem*{prop*}{\protect\propositionname}

\newtheorem{cor}[thm]{\protect\corollaryname}
\newtheorem*{cor*}{\protect\corollaryname}

\newtheorem{fact}[thm]{\protect\factname}
\newtheorem*{fact*}{\protect\factname}

\theoremstyle{definition}
\newtheorem{defn}[thm]{\protect\definitionname}
\newtheorem*{defn*}{\protect\definitionname}
\newtheorem{exam}[thm]{\protect\examplename}

\theoremstyle{remark}
\newtheorem{rem}[thm]{\protect\remarkname}
\newtheorem{claim}[thm]{\protect\claimname}

\numberwithin{equation}{subsection}

\newcommand{\cA}{\mathcal{A}}

\newcommand{\cC}{\mathcal{C}}
\newcommand{\cD}{\mathcal{D}}

\newcommand{\cK}{\mathcal{K}}
\newcommand{\cL}{\mathcal{L}}
\newcommand{\cM}{\mathscr{M}}
\newcommand{\cN}{\mathcal{N}}

\newcommand{\cP}{\mathcal{P}}

\newcommand{\R}{\mathbb{R}}

\newcommand{\N}{\mathbb{N}}

\newcommand{\E}{\mathbb{E}}
\newcommand{\eps}{\epsilon}

\newcommand{\indicator}[1]{\mathbbm{1}_{#1}}

\newcommand{\tensor}{\otimes}

\newcommand{\abs}[1]{\lvert#1\rvert}
\newcommand{\norm}[1]{\lvert\lvert#1\rvert\rvert}

\providecommand{\corollaryname}{Corollary}
\providecommand{\definitionname}{Definition}
\providecommand{\factname}{Fact}
\providecommand{\lemmaname}{Lemma}
\providecommand{\propositionname}{Proposition}
\providecommand{\theoremname}{Theorem}
\providecommand{\remarkname}{Remark}
\providecommand{\conjecturename}{Conjecture}
\providecommand{\examplename}{Example}
\providecommand{\claimname}{Claim}
\providecommand{\problemname}{Problem}
\providecommand{\questionname}{Question}

\newrefformat{prop}{Proposition \ref{#1}}
\newrefformat{subsub}{Section \ref{#1}}
\newrefformat{fact}{Fact \ref{#1}}
\newrefformat{cor}{Corollary \ref{#1}}
\newrefformat{exam}{Example \ref{#1}}
\newrefformat{rem}{Remark \ref{#1}}
\newrefformat{defn}{Definition \ref{#1}}
\newrefformat{conj}{Conjecture \ref{#1}}
\newrefformat{quest}{Question \ref{#1}}
\newrefformat{sub}{Section \ref{#1}}

\begin{document}

\title{Low Temperature Asymptotics in Spherical Mean Field Spin Glasses}

\author{Aukosh Jagannath}
\address[Aukosh Jagannath]{Courant Institute of Mathematical Sciences, 251 Mercer St.\ NY, NY, USA, 10012}
\email{aukosh@cims.nyu.edu}

\author{Ian Tobasco}
\address[Ian Tobasco]{Courant Institute of Mathematical Sciences, 251 Mercer St.\ NY, NY, USA, 10012}
\email{tobasco@cims.nyu.edu}

%

\begin{abstract}
In this paper, we study the low temperature limit of the spherical
Crisanti-Sommers variational problem. We identify the $\Gamma$-limit
of the Crisanti-Sommers functionals, thereby establishing a rigorous
variational problem for the ground state energy of spherical mixed
$p$-spin glasses. As an application, we compute moderate deviations
of the corresponding minimizers in the low temperature limit. In particular,
for a large class of models this yields moderate deviations for the
overlap distribution. We then analyze the ground state energy problem.
We show that this variational problem is dual to an obstacle-type
problem. This duality is at the heart of our analysis. We present
the regularity theory of the optimizers of the primal and dual problems.
This culminates in a simple method for constructing a finite dimensional
space in which these optimizers live for any model. As a consequence
of these results, we unify independent predictions of Crisanti-Leuzzi
and Auffinger-Ben Arous regarding the 1RSB phase in this limit. We
find that the ``positive replicon eigenvalue'' and ``pure-like''
conditions are together necessary for optimality, but that neither
are themselves sufficient, answering a question of Auffinger and Ben
Arous in the negative. We end by proving that these conditions completely
characterize the 1RSB phase in $2+p$-spin models.
\end{abstract}

\date{\today}

\maketitle

\section{Introduction}

In this paper, we study the Crisanti-Sommers variational problem which
is defined as follows. Let $\xi\left(t\right)=\sum_{p\geq2}\beta_{p}^{2}t^{p}$,
which we call the model, and assume that $\xi(1+\epsilon)<\infty$
for some $\epsilon>0$. The Crisanti-Sommers functional is defined
for $\mu\in\Pr\left([0,1]\right)$ by

\begin{equation}
\cP_{\beta,h,\xi}(\mu)=\frac{1}{2}\left(\int_{0}^{1}\beta^{2}\xi''(s)\hat{\mu}(s)ds+\int_{0}^{1}\left(\frac{1}{\hat{\mu}\left(s\right)}-\frac{1}{1-s}\right)\,ds+h^{2}\hat{\mu}\left(0\right)\right)\label{eq:CSfunc}
\end{equation}
where 
\[
\hat{\mu}\left(t\right)=\int_{t}^{1}\mu\left[0,s\right]\,ds.
\]
Here, $\beta$ is the inverse temperature and $h$ is the external
field and they satisfy $\beta>0$ and $h\geq0$. Note that since $\hat{\mu}(s)\leq1-s$,
the second integral is well-defined. The Crisanti-Sommers variational
problem is given by 
\begin{equation}
F(\beta,h,\xi)=\min_{\mu\in\Pr([0,1])}\frac{1}{\beta}\cP_{\beta,h,\xi}\left(\mu\right).\label{eq:Cris-som-vf}
\end{equation}
For experts: the functional defined above is a lower semi-continuous
extension of the functional originally described by Crisanti and Sommers
\cite{CriSom92}. Its minimization is the same as that of the functional
considered in \cite{TalSphPF06}. This is explained in more detail
in \prettyref{sub:csfunc-reform}.

The importance of the Crisanti-Sommers variational problem comes from
the study of spherical mixed $p$-spin glasses, which are defined
as follows. Let $\Sigma_{N}=S^{N-1}(\sqrt{N})$ and define the Hamiltonian
\[
H_{N}(\sigma)=\beta\sum_{p\geq2}^{\infty}\frac{\beta_{p}}{N^{\frac{p-1}{2}}}\sum_{i_{1},\ldots,i_{p}=1}^{N}g_{i_{1}\ldots i_{p}}\sigma_{i_{1}}\cdots\sigma_{i_{p}}+h\sum_{i=1}^{N}\sigma_{i}
\]
where $g_{i_{1}\dots i_{p}}$ are i.i.d.\ $\cN(0,1)$ random variables.
For the relationship between the study of these problems and the study
of the class of smooth, isotropic Gaussian processes on the sphere
in high dimension see \cite{ABA13}. Define the partition function
and Gibbs measure 
\[
Z_{N}=\int_{\Sigma_{N}}e^{H_{N}}dvol_{N}\quad\text{and}\quad G_{N}(dvol)=\frac{e^{H_{N}}}{Z_{N}}dvol_{N}
\]
where $dvol_{N}$ is the normalized volume measure on $\Sigma_{N}$.
It was predicted by Crisanti and Sommers \cite{CriSom92} and proved
by Talagrand \cite{TalSphPF06} and Chen \cite{ChenSph13} that the
thermodynamic limit of the free energy per site is given by the variational
formula 
\[
\lim\frac{1}{\beta N}\log Z_{N}=\min_{\mu\in\Pr([0,1])}\frac{1}{\beta}\cP_{\beta,h,\xi}(\mu).
\]
The minimizer, $\mu_{\beta,h,\xi}$, is thought of as the order parameter
in these systems, and is conjectured to be the limiting law of the
overlap, $R_{12}=(\sigma^{1},\sigma^{2})/N$, i.e.,
\[
\E G_{N}^{\tensor2}\left(R_{12}\in A\right)\to\mu_{\beta,h,\xi}(A).
\]
This is known, for example, when the collection $\{t^{p}:\beta_{p}\neq0\}$
is total in $C([0,1])$ (these are called \emph{generic }models)\cite{TalSphPF06,PanchSKBook}.

In this paper, we study the zero temperature, i.e., $\beta\to\infty$,
limit of the Crisanti-Sommers variational problem. This limit arises
naturally in the method of annealing, an important and nontrivial
technique used in the study of random optimization problems \cite{KirGelVec83,MMZ01,MPV87}.
The study of mean field spin glasses is intimately related to the
study of random optimization problems in highly disordered energy
landscapes. For such problems, it is important to determine the precise
asymptotics of the maximum in the limit $N\to\infty$, i.e., 
\[
\lim_{N\to\infty}\frac{1}{N}\max_{\sigma\in\Sigma_{N}}H_{N}(\sigma)=GSE.
\]
This quantity is called the \emph{ground state energy.} As an example,
the case where $\xi(t)=t^{2}$ and $h=0$ corresponds to the study
of the (renormalized) largest eigenvalue of a GOE random
matrix which is well understood. In contrast, the case $\xi(t)=t^{3}$
is much less understood, and corresponds to the study of the 
maximum of a random trilinear form with Gaussian coefficients. 
A related and natural question is to study the limiting law of the relative positions of near minimizers in the asympotic that their energies approach the ground state.
For a summary of what is known see \cite{ABA13,Sub15}.

The central idea of the method of annealing is that the ground state
energy can be computed from the free energy by sending the temperature
to zero. In our setting, this means that if $h=\beta h_{0}(\beta)$
and $h_{0}\to\overline{h}$ as $\beta\to\infty$, then 
\begin{equation}
GSE=\lim_{\beta\to\infty}\min_{\mu\in\Pr([0,1])}\frac{1}{\beta}\cP_{\beta,h,\xi}\left(\mu\right).\label{eq:GSE-is-lim-P}
\end{equation}
The proof of this result in our setting is standard. (See \cite{ABA13}
for a proof in the case that $h_{0}=0$. The case $h_{0}\neq0$ follows
by a straightforward extension of their arguments.) One expects in
the annealing limit that 
\[
\lim_{\beta\to\infty}\lim_{N\to\infty}\E\int f(R_{12})dG_{N}^{\tensor2}(R_{12}\in A)=f(1).
\]
As an application of our analysis, we obtain the next order correction
to this statement. 

Our approach to the zero temperature limit is through $\Gamma$-convergence.
This notion was introduced by de Giorgi and is a standard tool in
the asymptotic analysis of variational problems. An immediate consequence
of the theory that the ground state energy is the minimum value of
the $\Gamma$-limit of $\frac{1}{\beta}\cP_{\beta,h,\xi}$ . A further
consequence is that the minimizers at finite $\beta$ converge to
the minimizer at $\beta=\infty$, in an appropriate topology. By studying
the $\Gamma$-limit as a variational problem unto itself, we are able
to rule out certain conjectures pertaining to the character of the
minimizers at large, but finite, $\beta$.

The zero temperature problem is a strictly convex minimization problem.
In principle, one could study its minimizer through its first order
optimality conditions. This approach is well-known in the literature
surrounding the Parisi variational problem \cite{TalSphPF06,AuffChen13,JagTobPD15}.
In this paper, we take an entirely different approach through convex
duality. We obtain the convex dual of the zero temperature problem:
it is a concave maximization problem of obstacle-type. Obstacle-type
problems and their first order optimality conditions, called \textquotedblleft variational
inequalities\textquotedblright , have a long history in the calculus
of variations (see e.g. \cite{CafFried79,Caf98,KinderStamp2000}).
In these problems, the study of the contact set, the points at which
the obstacle and the optimizer are equal, is crucial. In our analysis,
we find an interesting connection between the contact set and the
choice of model. This connection shares similarities with some results
of Cimatti on the shape of a constrained elastic beam \cite{Cim73}.
Exploiting this connection, we are able to comment on the phase diagram
of the zero temperature problem in full generality.

Determining the full phase diagram of the Crisanti-Sommers variational
problem, particularly sharply determining phase boundaries,
remains an important and difficult question. The region of
$(\beta,h,\xi)$-space in which the minimizer is $1$-atomic is known
as the RS region, the region in which it is $(k+1)$-atomic is the
$k$RSB region, and the region in which it is not $k$-atomic for
any $k\in\N$ is the FRSB region. A typical question is to find explicit
conditions on $(\beta,h,\xi)$ that characterize a $k$RSB region.
In the FRSB region, little is known rigorously about the character
of the minimizer. If the minimizer is absolutely continuous on an
interval $[a,b]$, then on that interval its density is known (see
for example \cite{TalSphPF06,CriLeu04}). However, as suggested by
the work of \cite{Kra07,CriLeu07}, one expects that in full generality,
the support of the absolutely continuous part may consist of many
disjoint intervals. It is interesting
to ask if there is a systematic way to reduce the complexity of the
space in which the minimizer lives.

These questions have natural analogs when $\beta=\infty$.
Though difficult at the level of the primal problem,
they are very natural at the level of the dual: they are questions about the
topology of the contact set.  An isolated atom for the minimizer corresponds to
an isolated point in the contact set. An interval in the support of the minimizer
corresponds to an interval in the contact set. 
The question of ``how many RSBs" is then 
``how many connected components does the contact set have". We give 
a simple method to upper bound  the number of connected components. 
Furthermore, our work gives strong evidence for the predictions of \cite{Kra07,CriLeu07}. 
This scenario runs against the common intuition in the mathematical spin glass community. 

The bulk of this paper is regarding these questions. Through our analysis of the $\Gamma$-limit,
we describe a general algorithm for producing the minimizer which
reduces the problem to a finite dimensional optimization problem.
In the case of $2+p$ models, i.e.,\ $\xi(t)=\beta_{2}^{2}t^{2}+\beta_{p}^{2}t^{p}$,
we give an exact characterization of the 1RSB region in terms of the
coefficients of the model. For general models, these conditions are
seen to be necessary. This result rules out the (distinct) characterizations
of 1RSB suggested by \cite{CriLeu04} and \cite{ABA13}, and instead
proves that, in the case of $2+p$ models, the intersection of these
conditions characterizes 1RSB. Our main tool for establishing these
results is the new convex duality principle for the limiting functional
at $\beta=\infty$.

When a first draft of this paper was complete, we learned of the related
work of Chen and Sen \cite{ChenSen15} in which the authors also treated
the zero temperature limit of the Crisanti-Sommers variational problem.
In \cite{ChenSen15}, the authors provided an alternative, but equivalent,
variational representation for the ground state energy, obtained the
first order optimality conditions and its immediate consequences,
then turned to probabilistic questions which, while related, do not
overlap with the present work. See \prettyref{rem:chen-sen-rem} and
the discussion after \prettyref{thm:duality} for more on the relation
between these results.

\subsection{Limiting Problem: Gamma convergence results}

Our first result establishes a variational representation for the
ground state energy. We begin by introducing the topological space 

\begin{equation}
\cA=\left\{ \nu\in\cM\left(\left[0,1\right]\right):\ \nu=m\left(t\right)\,dt+c\delta_{1},\ m\left(t\right)\geq0\text{ is non-decreasing and cadlag}\right\} \label{eq:calA}
\end{equation}
equipped with the relative topology induced by the weak-$*$ topology
on $\cM([0,1])$, the space of finite measures on $[0,1]$, i.e.,
the topology of weak convergence of measures. In the subsequent, $m(t)$
will always refer to the unique representative of the density of $\nu$
that satisfies the above conditions and is left-continuous at $1$. 

On the space $\cA$, we define the subsets 
\[
X_{\beta}=\left\{ \nu\in\cA\ :\ d\nu=\beta\mu\left[0,t\right]dt,\quad\mu\in\Pr\left[0,1\right]\right\} ,
\]
and we lift the functional $\frac{2}{\beta}\cP_{\beta,h,\xi}$ to
$\cA$ as $F_{\beta,h,\xi}:\cA\to[0,\infty]$, 
\[
F_{\beta,h,\xi}(\nu)=\begin{cases}
\frac{2}{\beta}\cP_{\beta,h,\xi}(\mu) & \nu\in X_{\beta}\\
\infty & \nu\notin X_{\beta}
\end{cases}.
\]
Finally, we define the functional $GS_{h,\xi}:\cA\to[0,\infty]$ as

\begin{equation}
GS_{h,\xi}\left(\nu\right)=\begin{cases}
\int_{0}^{1}\xi''\left(s\right)\nu[s,1]+\frac{1}{\nu[s,1]}\,ds+h^{2}\nu[0,1] & \nu\neq0\\
\infty & \nu=0
\end{cases}.\label{eq:GSfunc}
\end{equation}

Observe that $GS$ has a unique minimizer. Indeed, it is strictly
convex by the strict convexity of $x\to\frac{1}{x}$ and sequentially
lower semi-continuous by Fatou's lemma. Furthermore, the sets $\left\{ GS(\nu)\leq C\right\} $
for $C<\infty$ are sequentially compact in $\cA$ by \prettyref{lem:A-set-cpt}
applied with $f=\xi''$. 

Before we state our first result, we remind the reader of the notion
of \emph{sequential $\Gamma$-convergence} \cite{Bra02}.
\begin{defn}
Let $X$ be topological space. We say that a sequence of functionals
$F_{n}:X\to[-\infty,\infty]$ \emph{sequentially $\Gamma-$converges}
to $F:X\to[-\infty,\infty]$ if
\begin{enumerate}
\item The $\Gamma-\liminf$ inequality holds: for every $x$ and every sequence
$\lim_{n\to\infty}x_{n}=x$ 
\[
\liminf_{n}F_{n}(x_{n})\geq F(x)
\]

\item The $\Gamma-\limsup$ inequality holds: for every $x$ there is a
sequence $\lim_{n\to\infty}x_{n}=x$ such that
\[
\limsup_{n\to\infty}F_{n}(x_{n})\leq F(x)
\]

\end{enumerate}
We denote this by $F_{n}\stackrel{\Gamma}{\to}F$. For a sequence
of functionals indexed by a real parameter $\beta$ we say that $F_{\beta}\stackrel{\Gamma}{\to}F$
if for every subsequence $\beta_{n}\to\infty$, $F_{\beta_{n}}\stackrel{\Gamma}{\to}F$.\end{defn}
\begin{rem}
The following remark will only be of interest to experts in the field
of $\Gamma$-convergence. We observe here that the notion of sequential
$\Gamma$-convergence is distinct from the notion of $\Gamma$-convergence
in our setting as we are not working in a metrizable space. For a
brief discussion of this see \cite{Bra02,Dal93}. Nevertheless, the
usual consequences of $\Gamma$-convergence carry through to our setting
in the sequential case. We place the proof of those results that we
use in the appendix.\end{rem}
\begin{thm}
\label{thm:gamma-conv-min-conv} Suppose that $(\beta,h,\xi_{\beta})$
are such that $\frac{h}{\beta}\to\bar{h}$ and $\xi''_{\beta}\to\xi''$
uniformly as $\beta\to\infty$. Then, 
\[
F_{\beta,h,\xi_{\beta}}\stackrel{\Gamma}{\to}GS_{\bar{h},\xi}.
\]
 In particular, 
\begin{equation}
GSE=\min_{\nu\in\cA}\,\frac{1}{2}GS_{h,\xi}\left(\nu\right).\label{eq:GSE-vf}
\end{equation}
 Furthermore, we have that if $\nu_{\beta}$ are the (unique) minimizers
of $F_{\beta,h,\xi_{\beta}}$ then 
\[
\nu_{\beta}\stackrel{w-*}{\longrightarrow}\nu
\]
 where $\nu$ is the unique minimizer of $GS_{h,\xi}$.
\end{thm}
As explained above, an an immediate corollary of the $\Gamma-$convergence
is a moderate deviation principle in the limit $\beta\to\infty$ for
the minimizers.
\begin{cor}
\label{cor:moderate-dev}Let $\nu=m(t)dt+c\delta_{1}\in\cA$ be the
unique minimizer of $GS_{h,\xi}$, and $\mu_{\beta}$ be the (unique)
minimizers of $\cP_{\beta,h,\xi_{\beta}}$ where $(\beta,h,\xi_{\beta})$
satisfy the conditions of \prettyref{thm:gamma-conv-min-conv}. Then, 
\begin{enumerate}
\item For any $f\in C^{1}$, 
\[
\lim\beta\left[f(1)-\int fd\mu_{\beta}\right]=\int f'd\nu.
\]

\item For every $t<1$ that is a continuity point of $m(t)$, 
\[
\beta\mu_{\beta}[0,t]\to m(t).
\]

\item Let $q_{\beta}\to1$ be such that $\beta\mu_{\beta}[0,q_{\beta})\to m(1^{-})$,
and suppose that $m(1^{-})<\infty$. Then if $Y_{\beta}$ have law
$\mu_{\beta}$, it follows that 
\begin{align*}
\nu\left(\left\{ 1\right\} \right) & =\lim\E_{\mu_{\beta}}\left(\beta(1-Y_{\beta})\vert Y_{\beta}\in[q_{\beta},1]\right).
\end{align*}

\end{enumerate}
\end{cor}

\begin{rem}
\label{rem:chen-sen-rem}In \cite{ChenSen15}, Chen and Sen also obtained
\eqref{eq:GSE-vf}. In their notation $\nu([0,1])=L$ and $m(t)=\alpha_{0}(t)$.
The convergence results stated there are equivalent to \prettyref{cor:moderate-dev}
(2), combined with the convergence of $\nu_{\beta}\left([0,1]\right)$
to $\nu([0,1])$. The strict inequality between $L$ and $\int\alpha_{0}$
in \cite{ChenSen15} will follow from \prettyref{thm:(a-priori-estimates)}.
\end{rem}

\begin{rem}
An immediate consequence of this result is a moderate deviation principle
for the overlap distribution, $R_{12}$, for models for which $\mu_{\beta}$ is
known to be its limiting law (e.g. generic models).
By a standard differentiation and convexity argument (see, e.g.,  \cite[Theorems 3.7, 3.8]{PanchSKBook}), the Gibbs measure concentrates on the set $\{\abs{H_N/N-\E\langle H_N\rangle/N}<\epsilon\}$ in the thermodynamic limit for all $\eps>0$. By a standard integration by parts argument (see also \cite[Theorem 1.2]{TalSphPF06}), $\lim_{N}\E\langle H_N \rangle/N = \int(\xi(1)-\xi(t))\beta d\mu_\beta$. These results, along with those in \prettyref{cor:moderate-dev}, yield asymptotic information about the law of the relative positions of near maximizers of $H_N$ in the large $N$, large $\beta$ limit.
\end{rem}

\begin{rem}
As we shall soon see, $m(1^{-})<\infty$ in our setting (see \prettyref{thm:(a-priori-estimates)}).
We note, however, that it is not true that $\beta\mu_{\beta}\to dm$
weakly as measures, since $\beta\mu_{\beta}([0,1])=\beta\to\infty$.
\end{rem}

\begin{rem}
One has to be careful interpreting $(3)$ for the following reason:
it may be that $\beta(1-q_{\beta})$ explodes. This is neither prevented
by the convergence on $\cA$ nor by finite energy considerations.
The following is an interesting example to keep in mind. Let $\beta\mu_{\beta}=m\delta_{q_{1}}+(1-m)\delta_{q_{2}}$
where $q_{1}=1-\frac{1}{\sqrt{\beta}}$, $q_{2}=1-\frac{1}{\beta}$,
and $m=\sqrt{\beta}$. Then $\int_{s}^{1}\beta\mu_{\beta}dt\sim2+O(\frac{1}{\sqrt{\beta}})$
and the corresponding energy scales like $\frac{1}{2}+\frac{1}{\sqrt{\beta}}\log(2)+O(\frac{1}{\beta})$
which is finite. A similar example can be constructed to show that
a quantification of the rate at which $q_{*}(\beta)=\sup supp\mu_{\beta}\to1$
is out of the reach of these methods as the topology of these results
are too weak (though clearly $\limsup\beta(1-q_{*}(\beta))\leq\limsup\int_{0}^{1}\beta\mu_{\beta}([0,t])dt\leq C$).
We believe that one would require sharp estimates on $q_{*}$, such
as might come from the first order optimality conditions, in order
to obtain such a result.
\end{rem}

\begin{exam}
Let $h=0$. Suppose that the minimizers $\mu_{\beta}$ are $1RSB$.
Then we have that $\beta(1-q_{*})\to\nu\left\{ 1\right\} $ and $\beta\mu_{\beta}\left(\{0\}\right)\to\nu([0,1))$. 
\end{exam}
The proofs of these results are in \prettyref{sec:Gamma-convergence-results}.

\subsection{Convex Duality Results}

We turn now to the analysis of the limiting variational problem. First,
we find it convenient to make the following change of variables. Define

\begin{align}
\cC & =\left\{ \phi\in C([0,1]):\ \phi\geq0,\ \phi\text{ is non-increasing},\ \phi\text{ is concave}\right\} \label{eq:calC}
\end{align}
and define the functional $P_{h,\xi}:\cC\to[0,\infty]$ by
\begin{equation}
P_{h,\xi}\left(\phi\right)=\int\xi''\phi+\frac{1}{\phi}dx+h^{2}\phi(0).\label{eq:Pfunc}
\end{equation}
Observe that the space $\cA$ and the space $\cC$ are in one-to-one
correspondence. In particular, every $\phi\in\cC$ is of the form
$\nu[s,1]$ for some $\nu\in\cA$ and similarly every $\nu\in\cA$
is of the form $\phi'(t)dt+\phi(1)\delta_{1}$ for some $\phi\in\cC$.
Here, the derivative $\phi'$ is understood distributionally and is
an element of $L^{1}$. This correspondence and other important results
about these spaces are summarized in \prettyref{sub:On-the-sets-C-and-A}.
In particular, observe that $GS(\nu)=P(\phi)$ whenever $\nu$ and
$\phi$ are in correspondence, so that 
\begin{equation}
GSE=\min_{\phi\in\cC}\,\frac{1}{2}P_{h,\xi}\left(\phi\right).\label{eq:P-GSE}
\end{equation}

Our main tool in the analysis of the limiting variational problem
is an important duality principle, which relates this problem to a
one-dimensional variational problem of obstacle-type. To define it,
let 

\begin{equation}
\cK_{h,\xi}=\left\{ \eta\in C([0,1])\ :\ \eta\geq\xi,\ \eta\left(1\right)=\xi\left(1\right),\ \eta'\left(0\right)=\xi'\left(0\right)-h^{2},\ \eta\text{ is convex}\right\} ,\label{eq:K-xi}
\end{equation}
equipped with the norm topology. Basic properties of this space are
summarized in \prettyref{sub:On-the-sets-C-and-A}. In particular,
$\eta'\in BV((0,1))$ so that $\eta'$ has well-defined trace at $0$.
Furthermore, $\eta''$ can be uniquely represented by $\mu\in\cM([0,1])$
with $\mu(\{0,1\})=0$. In the following, $\eta''$ will always refer
to this representative. 

Now define the functional $D:\cK_{h,\xi}\to[0,\infty)$, 
\begin{equation}
D\left(\eta\right)=2\int\sqrt{\eta_{ac}''(x)}dx.\label{eq:Dfunc}
\end{equation}
Here for a measure $\nu$, we let $\nu_{ac}(x)=\frac{\partial\nu}{\partial\cL}$
be its density with respect to $\cL$, the Lebesgue measure. Basic
properties of $D$ are proved in \prettyref{sub:Square-Roots-of}.
In particular, by \prettyref{cor:D-usc}, it is upper semi-continuous.

We then have the following duality principle:
\begin{thm}
\label{thm:duality}We have that
\[
\min_{\phi\in\cC}\,P_{h,\xi}\left(\phi\right)=\max_{\eta\in\cK_{h,\xi}}\,D\left(\eta\right).
\]
Furthermore, the optimizers satisfy 
\[
\phi^{2}\eta''(dx)=dx.
\]

\end{thm}
These results are proved in \prettyref{sec:The-Dual-of}. Problems
of the type 
\begin{equation}
GSE=\max_{\eta\in\cK_{h,\xi}}\,D\left(\eta\right)\label{eq:D-is-GSE}
\end{equation}
are called \emph{obstacle problems} and have a rich literature. The
obstacle problem approach to studying variational problems on the
space of measures has become an important tool, see for example \cite{Ser15}. 

Before turning to the analysis of the primal-dual pair $(P,D)$, we
wish to comment briefly on the relationship between our approach to
optimality and that which concerns the primal problem alone. Consider
the first order optimality conditions for the primal problem at $h=0$:
$\phi\in\cC$ is optimal if and only if 
\begin{equation}
(\xi''-\frac{1}{\phi^{2}},\psi-\phi)\geq0\quad\forall\psi\in\cC.\label{eq:primal_opt}
\end{equation}
This variational inequality and others like it play an essential role
in the analysis of Parisi measures (see e.g. \cite{TalSphPF06,AuffChen13,JagTobPD15,ChenSen15}).
In this approach, the difficulty is to prove that a certain function
on $[0,1]$ depending on the choice of measure is minimized on the
support of said measure. (In our work, $dm=-\phi''$ plays the role
of this measure.) The duality between $P$ and $D$, however, suggests
an entirely different approach, namely the simultaneous analysis of
the variables $\phi$ and $\eta$. The optimal pair, $(\phi,\eta)$,
not only achieves the equality $P(\phi)=D(\eta)$, it is characterized
by it. Furthermore, the variational inequality \prettyref{eq:primal_opt}
is implied by this observation. This simultaneous analysis is the
crux of our approach.

We now present an analysis of the optimizers. We begin by discussing
their regularity. 
\begin{thm}
\label{thm:(a-priori-estimates)} Let $\phi\in\cC$ and $\eta\in\cK_{h,\xi}$
be optimal for $P_{\xi}$ and $D$ respectively. Then, 
\begin{enumerate}
\item For $\phi$:

\begin{enumerate}
\item There is a $c>0$ such that $0<c\leq\phi$.
\item $\phi'\in L^\infty$.
\end{enumerate}
\item For $\eta:$

\begin{enumerate}
\item $\eta\in C^{2}([0,1])$
\item $\frac{1}{\sqrt{\eta''}}\in C([0,1])$ , has distributional derivative
$\left(\frac{1}{\sqrt{\eta''}}\right)'\in L^{1}$ which is monotone
decreasing, and has second distributional derivative 
\[
\left(\frac{1}{\sqrt{\eta''}}\right)''=-\mu
\]
for some non-negative Radon measure on $(0,1)$. Furthermore $supp\mu\subset\left\{ \xi=\eta\right\} .$ 
\item On the set $\{\xi=\eta\}$, 

\begin{enumerate}
\item $\eta'(t)=\xi'(t)$
\item $\eta''(t)\geq\xi''(t)$
\end{enumerate}
\item (Natural Boundary Conditions) We have that

\begin{enumerate}
\item $\eta'(1)=\xi'(1)$
\item $\eta(0)=\xi(0)$ or $\phi'(0)=0$
\end{enumerate}
\end{enumerate}
\item For $\mu:$ we have that 
\[
\phi''=\left(\frac{1}{\sqrt{\eta''}}\right)''=-\mu
\]
as elements of $\cD'$. Moreover, $\mu$ is a finite measure.
\end{enumerate}
\end{thm}
Since the dual problem is an obstacle-type problem, the following
definition is natural. 
\begin{defn}
A point $t\in[0,1]$ such that $\eta(t)=\xi(t)$ is called a \emph{contact
point. }The set of contact points is called the \emph{contact set}.
\end{defn}
We then have the following regularity result concerning the contact
set of this obstacle problem.
\begin{thm}
\label{thm:wiggles-principle}Let 
\[
\mathfrak{\mathfrak{d}}(t)=\left(\frac{1}{\sqrt{\xi''}}\right)''(t).
\]
Then we have the following two cases:
\begin{enumerate}
\item If $\mathfrak{d}(t)>0$ on $\left(a,b\right)$, then there are at
most two contact points in $\left[a,b\right]$.
\item If $\mathfrak{d}(t)\leq0$ on $[a,b]$, then if there are two contact
points $t_{1},t_{2}\in\left[a,b\right]$, then $\left[t_{1},t_{2}\right]\subset\left\{ \eta=\xi\right\} $.
\end{enumerate}
\end{thm}
\begin{rem}
The finite temperature analogue of part 1 of this result can be seen
in \cite{CriLeu04} as observed in \cite{TalSphPF06}.
\end{rem}
This result provides us with a systematic dimension reduction which
reduces the analysis of \prettyref{eq:P-GSE} to a finite dimensional
optimization problem. Rather than describing this at the level of
generalities, we prefer to illustrate these ideas through three examples.
In each, we think of building up the ansatz on connected components
of the sets 
\[
N=\{t\in[0,1]:\mathfrak{d}\leq0\}\quad\text{and}\quad P=\{t\in[0,1]:\mathfrak{d}\geq0\}.
\]
For simplicity, we work with $h=0$.
\begin{exam}\label{exam:4wiggles}
Let $\xi(t)=\frac{300}{601}t^{2}+\frac{200}{601}t^{4}+\frac{100}{601}t^{15}+\frac{1}{601}t^{60}$.
The corresponding $\mathfrak{d}$ is depicted in \prettyref{fig:4wiggles}.
\begin{figure}
\includegraphics[width=0.4\textwidth]{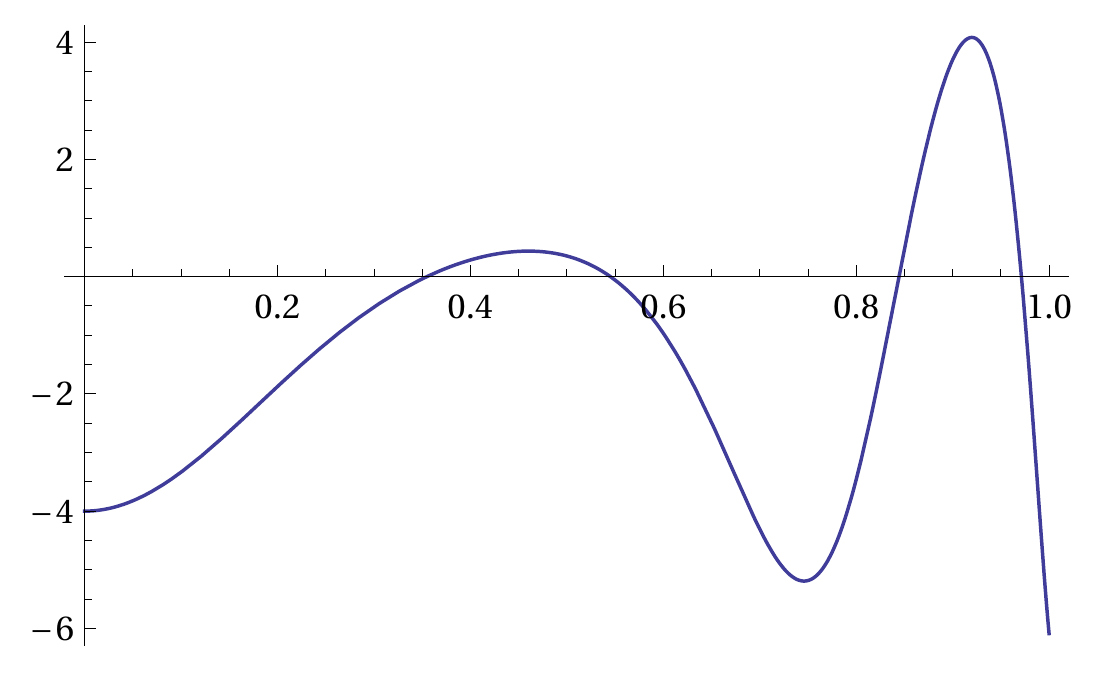}\caption{This plot shows the function $\mathfrak{d}$ from \prettyref{thm:wiggles-principle} for the model $\xi(t)=\frac{300}{601}t^{2}+\frac{200}{601}t^{4}+\frac{100}{601}t^{15}+\frac{1}{601}t^{60}$. The intervals on which $\mathfrak{d}$ is negative are compatible with FRSB. The intervals on which $\mathfrak{d}$ is positive permit at most two atoms each. As shown in \prettyref{exam:4wiggles}, this can be used to reduce the primal problem to a certain finite dimensional optimization problem, that depends on the particular arrangement of these intervals.}
\label{fig:4wiggles}
\end{figure}
 Evidently, $\mathfrak{d}(t)=0$ has exactly four solutions in $[0,1]$,
which we label as $0<r_{1}<r_{2}<r_{3}<r_{4}<1$. We call $r_{0}=0$
and $r_{5}=1$. Now $N=[r_{0},r_{1}]\cup[r_{2},r_{3}]\cup[r_{4},r_{5}]$
and $P=[r_{1},r_{2}]\cup[r_{3},r_{4}]$. By \prettyref{thm:wiggles-principle},
each component of $N$ can intersect the contact set in at most one
closed interval (possibly degenerate). Each component of $P$ can
intersect the contact set in at most two points. This motivates the
following construction. Let $q_{i,j}\in[r_{i-1},r_{i}]$ and $m_{i,j}\in[0,\infty)$
for $i=1,\dots,4$, $j=1,2$. Assume that $q_{i,1}\leq q_{i,2}$ for
all $i$. Then, by the results above, the optimal $\phi$ must be
of the form $\phi=\int_{t}^{1}\mu[0,s]ds+c$ where 
\[
\mu=\sum_{i}m_{i,1}\delta_{q_{i,1}}+m_{i,2}\delta_{q_{i,2}}-\sum_{i:I_{i}\subset N}\mathfrak{d}\cdot\indicator{[q_{i,1},q_{i,2}]}dx
\]
and $c\in(0,\infty)$. In fact, $q_{1,1}=0$ by the same argument
as in the proof of \prettyref{lem:atomatzero}. 
\end{exam}
The next examples are comparatively straightforward. The reader will
observe that the key simplification comes from the fact that either
$N=\emptyset$ or $P=\emptyset$.
\begin{exam}
Let $\xi(t)=\sinh(t)$. Then $\mathfrak{d}(t)=8(5+\cosh2t)(\sinh t)^{-5/2}$
and $P=[0,1]$. Hence, by \prettyref{thm:wiggles-principle}, there
can be at most two contact points in $[0,1]$. We already know that
$1$ is a contact point, and by the same argument as in the proof
of \prettyref{lem:atomatzero} we can show that $0$ must be a contact
point as well. Therefore the optimal $\phi$ must be of the form $\phi=m(1-t)+c$
where $m,c>0$. (Since $\xi\neq\xi_{RS}$ we know that $m\neq0$ by
\prettyref{lem:RSiffSK}.)
\end{exam}

\begin{exam}
Let $\xi(t)=\frac{14}{15}t^{2}+\frac{1}{15}t^{4}$. Then $\mathfrak{d}(t)=\frac{3\sqrt{15}}{2}(6t^{2}-7)(3t^{2}+7)^{-5/2}$
and $N=[0,1]$. Hence, by \prettyref{thm:wiggles-principle}, if there
are two contact points $a,b\in[0,1]$ then $[a,b]\subset\{\eta=\xi\}$.
Now $1$ is a contact point, and by the same argument as in the proof
of \prettyref{lem:atomatzero} we can show that $0$ must be a contact
point as well. Therefore the optimal $\eta$ must be $\eta=\xi$.
\end{exam}
The proofs of these results are given in \prettyref{sec:The-Dual-of}
and \prettyref{sec:Regularity-theory-for}.

\subsection{Application to the Analysis of Phase Transitions}

The notions of RS, RSB, kRSB, and FRSB have natural extensions to
$\beta=\infty$. In this section, we define these extensions and apply
the methods of the previous section to study the 1RSB class in detail.
In the case of $2+p$ models, we characterize 1RSB exactly. For simplicity,
we will \textbf{assume} that $h=0$ throughout the remainder of the introduction.
We use the shorthand $P_{\xi}=P_{0,\xi}$ and $\cK_{\xi}=\cK_{0,\xi}$.

\subsubsection{Definition of kRSB}

The following is an elementary consequence of the $\Gamma$-convergence.
\begin{lem}
Let $(\beta,h,\xi_{\beta})$ satisfy the assumptions of \prettyref{thm:gamma-conv-min-conv},
and assume that $\frac{h}{\beta}=o_{1}(\beta)$. Suppose that there
are $k\in\N$, $\beta_{c}\in\R_{+}$ such that for all $\beta\geq\beta_{c}$,
the minimizer $\mu_{\beta}$ of $\cP_{\beta,h,\xi_{\beta}}$ is $k$
atomic. Then the minimizer $\nu=m(t)dt+c\delta_{1}$ of $GS_{0,\xi}$
is such that $dm$ is at most $k-1$ atomic on $[0,1)$.
\end{lem}
With this and the correspondence $\cA\leftrightarrow\cC$ in mind,
we define for each $k\in\N$ the set 
\[
RSB_{k}=\{\phi\in\cC:dm\ \text{is}\ k-\text{atomic on}\ [0,1)\}
\]
and, with slight abuse of notation, we call $kRSB$ the set of models
such that the optimal $\phi$ for $P_{\xi}$ is in $RSB_{k}$. Similarly,
we call $RS$ the set of models such that the optimal $\phi$ is constant.
We call $RSB$ the complement of this, and we call $FRSB$ the region
where the optimal $dm$ is neither zero nor $k$-atomic for any $k\in\N$. 

In the ground state problem at zero external field, RS is particularly
simple.
\begin{lem}
\label{lem:RSiffSK} $RS=\{\xi_{SK}\}$.\end{lem}
\begin{proof}
Applying the natural boundary conditions from \prettyref{thm:(a-priori-estimates)}
to $\eta$, we see that $\phi=c$ yields $c^{2}=\frac{1}{\xi'(1)}$,
so that $\eta$ must be of the form 
\begin{equation}
\eta(t)=\xi(1)+\xi'(1)\frac{t^{2}-1}{2}.\label{eq:eta_SK}
\end{equation}
By the same regularity theorem, $\eta''(1)\geq\xi''(1)$ so that $\xi'(1)\geq\xi''(1).$
This implies that $\xi=\xi_{SK}$. On the other hand, if $\xi=\xi_{SK}$,
then $\eta$ given by\prettyref{eq:eta_SK} is in $\cK_{h,\xi}$ and
maximizes $D$ so that the optimal $\phi$ is constant.
\end{proof}
The next result establishes the existence of an ``atom at zero''
at $\beta=\infty$.
\begin{lem}
\label{lem:atomatzero} If $\xi\in kRSB$, then the optimal $\phi$
for $P_{\xi}$ satisfies $\phi'(0)>0$. Equivalently, $dm$ has an
atom at zero.\end{lem}
\begin{proof}
Suppose that $q\in\left(0,1\right)$ is an atom of $dm$, $\phi''=0$
in $\left(0,q\right)$, and $\phi'\left(0\right)=0$. Then $q$ is
a contact point of $\eta$, and $\phi$ is constant on $(0,q)$. Hence,
by \prettyref{thm:(a-priori-estimates)}, the optimal $\eta$ is of
the form $\eta\left(t\right)=a+bt+ct^{2}$ for $t\in\left(0,q\right)$.
By that same theorem, $\eta'(0)=\xi'(0)$, $\eta\left(q\right)=\xi\left(q\right)$,
and $\eta'\left(q\right)=\xi'\left(q\right)$. Thus,
\[
\eta(t)=\xi(q)-\frac{\xi'(q)q}{2}+\frac{\xi'(q)}{2q}t^{2}\quad t\in(0,q).
\]
However, since $q$ is a contact point, $\xi''\left(q\right)\leq\eta''\left(q\right)$
and hence $\xi''(q)\leq\frac{\xi'(q)}{q}.$ This can only happen if
$\xi=\xi_{SK}$.
\end{proof}

\subsubsection{Characterization of 1RSB \label{sub:Characterization-of-1RSB}}

In this section, we will study the special case that the minimizer
is 1RSB. We re-define $RSB_{1}$ to be 

\begin{equation}
RSB_{1}=\left\{ \phi\ :\ \phi\left(t\right)=c+m\left(1-t\right),\ m,c\in(0,\infty)\right\} .\label{eq:RSB1}
\end{equation}
This slight abuse of notation is justified by \prettyref{lem:atomatzero}
since we are at zero external field. We will refer only to this definition
for this section.

The 1RSB region has been studied in the physics and mathematics literature
through different techniques. This lead to different proposed characterizations
of the 1RSB region. Auffinger and Ben Arous conjectured a criterion
for when a model should be $1RSB$, which they call \emph{pure-like}
\cite{ABA13}. Let $\nu=\xi(1)=1$, $\nu'=\xi'(1)$ and $\nu''=\xi''(1)$,
and define 
\begin{equation}
ABA(\xi)=ABA(\nu',\nu'')=\log\left(\frac{\nu''}{\nu'}\right)-\left(\frac{(\nu''-\nu')(\nu''-\nu'+\nu'^{2})}{\nu''\nu'^{2}}\right).\label{eq:ABA-func}
\end{equation}

\begin{defn}
A model $\xi$ is called \emph{pure-like} if $ABA>0$, \emph{critical}
if $ABA=0$, and \emph{full-like} if $ABA<0$.
\end{defn}
Separately, Crisanti and Leuzzi \cite{CriLeu04} predicted that the
model is 1RSB provided the 1RSB Replicon Eigenvalue is positive. To
define this, we introduce for $\phi\in\cC$ the \emph{formal conjugate}
\begin{equation}
\eta\left(t\right)=\xi\left(1\right)-R\left(t\right)\label{eq:formal-conj-1}
\end{equation}
where
\begin{equation}
R\left(t\right)=\int_{t}^{1}\int_{0}^{s}\frac{1}{\phi^{2}\left(\tau\right)}d\tau ds.\label{eq:R}
\end{equation}
(In the case that $h\neq0$, this formula would have an extra term.)
\begin{defn}
\label{def:repliconeigenvalue}Let $\xi\neq\xi_{SK}$ and let $m,c>0$
solve \eqref{eq:FP1RSB}. Let $\eta$ be the formal conjugate to $\phi=m(1-t)+c$.
The quantity $\eta''(0)-\xi''(0)$ is called the \textbf{1RSB Replicon
eigenvalue }or simply the \textbf{replicon eigenvalue}.\textbf{ }
\end{defn}
A consequence of \prettyref{thm:(a-priori-estimates)} is that in
order for a model to be 1RSB, it must be both pure-like and have positive
replicon eigenvalue.
\begin{cor}
\label{cor:consistency1RSB}Suppose that the optimal $\phi\in RSB_{1}$.
Then necessarily,
\begin{enumerate}
\item $m,c$ solve \eqref{eq:FP1RSB};
\item The model is pure-like or critical: $\eta''\left(1\right)\geq\xi''\left(1\right)$;
\item The replicon eigenvalue is non-negative: $\eta''\left(0\right)\geq\xi''\left(0\right)$.
\end{enumerate}
\end{cor}
\begin{rem}
Let $(\beta,h,\xi_{\beta})$ satisfy the assumptions of \prettyref{thm:gamma-conv-min-conv},
and assume that $\frac{h}{\beta}=o_{1}(\beta)$. Then as a consequence of this result and the convergence of minimizers
from the $\Gamma$-convergence, we see that both conditions (2) and (3) must be met for $(\beta,h,\xi_\beta)$ to
be 1RSB for large $\beta$.
\end{rem}
That being said, we find that in full generality, neither of these conditions are themselves sufficient for optimality. For example,
models of the form $\xi(t)=\mu t^{2}+(1-\mu)t^{4}$ are pure-like
provided $\mu<\mu_{c}\approx.786444$. However, for the choice $\mu=.7,$
the formal conjugate $\eta$ to $\phi=m(1-t)+c$, where $m,c$ solve \eqref{eq:FP1RSB}, satisfies $\eta(t)-\xi(t)<0$
for $t<.4$. Hence, this model is not 1RSB. Similarly, a model
satisfying $\xi''\left(0\right)=0$ always has non-negative replicon
eigenvalue. Therefore it suffices to find a model with no $p=2$ term
which is full-like. As described in \cite[Fig.\ 2]{ABA13}, there
exists $\mu\in(0,1)$ such that $\xi(t)=\mu t^{4}+(1-\mu)t^{30}$
is full-like. 

The analysis of 1RSB in the specific case of $2+p$ models is of particular
interest to the spin glass community. These are models of the form
$\xi(t)=\mu t^{2}+(1-\mu)t^{p}$, $\mu\in[0,1]$. In this setting,
we resolve the 1RSB conjectures. 
\begin{thm}
\label{thm:CL-ABA-2-p} Let $\xi$ be a $2+p$ model other than $\xi_{SK}$.
Then $\xi\in1RSB_{\infty}$ if and only if both the replicon eigenvalue
is non-negative and the model is pure-like or critical.
\end{thm}

\begin{rem}
It is interesting to note that the essential difficulties in proving
a result of this type, namely proving the obstacle condition, bears
a striking resemblance to testing the validity of the second moment
method approach of Subag \cite{Sub15}. 
\end{rem}
We close this introduction by noting the following curiosity: one
may be tempted to conjecture that the result of \prettyref{thm:CL-ABA-2-p}
holds for general models. Such a result would have to rest crucially
on the assumption that the power series $\xi$ has non-negative coefficients.
To see this, let $\xi$ be a 1RSB model and let $\phi$ be optimal
for $P$. Of course, $\phi\in RSB_{1}$, and its formal conjugate
$\eta$ is optimal for $D$. Note that none of the arguments leading
up to \prettyref{eq:FP1RSB} require that $\xi$ is a power series
with positive coefficients. Furthermore, only $\xi(1)$ and $\xi'(1)$
are required to determine $m_{*},c_{*}$ by \eqref{eq:FP1RSB}. Thus,
if we change $\xi$ by adding a bump function which is supported away
from $0$ and $1$, $\phi$ must still be the $RSB_{1}$ ansatz as
above. Evidently, we can arrange for the altered $\xi$ to not satisfy
the obstacle condition $\eta\geq\xi$. Thus, the positivity of the
coefficients in the power series of $\xi$ is crucial to the validity
of theorems of the form \prettyref{thm:CL-ABA-2-p}. 

The results from this section are proved in \prettyref{sec:1RSB}.

\subsection{Acknowledgments }

We would like to thank our advisors G. Ben Arous and R.V. Kohn for
their support. A.J. would like to thank E. Subag and O. Zeitouni for
helpful discussions regarding the presentation of this work. We would
like to thank the New York University GRI Institute in Paris for its
hospitality during the initial phase of this project. A.J. would like
to thank the Department of Mathematics at Northwestern University
for its hospitality during the preparation of this paper. This research
was conducted while A.J. was supported by a National Science Foundation
Graduate Research Fellowship DGE-0813964; and National Science Foundation
grants DMS-1209165 and OISE-0730136, and while I.T. was supported
by a National Science Foundation Graduate Research Fellowship DGE-0813964;
and National Science Foundation grants OISE-0967140 and DMS-1311833.

\subsection{Notation and Spaces}

The notation $\fint_{I}f$ denotes the average of $f$ over $I$ .
$\partial_{+}f$ and $\partial_{-}f$ are respectively the right and
left derivative. $C_{loc}^{k}((0,1))$ is the space of functions that
are $C^{k}$ on every compact subset of $(0,1)$ and $W^{k,p}((0,1))$
is the Sobolev space of functions on $(0,1)$ that are $k$ times
weakly differentiable with derivatives lying in $L^{p}$. The space
$\cD=C_{c}^{\infty}((0,1))$ is the usual space of test functions
and $\cD'$ denotes the space of distributions. $BV=BV((0,1))$ is
the space of bounded variation functions, i.e., those $f\in L^{1}$
with first distributional derivative $f'$ given by finite signed
measures on $(0,1)$. $\cM=\cM([0,1])$ is the space of finite signed
measures on $[0,1]$. $\cM_{+}\subset\cM$ are those measures that
are non-negative.

\section{Gamma convergence results \label{sec:Gamma-convergence-results}}

We begin this section by proving \prettyref{thm:gamma-conv-min-conv}.
We then turn to proving \prettyref{cor:moderate-dev}.

\subsection{The functional convergence }

In the following two lemmas we fix $\xi_{\beta}=\xi$ and $h=0$ and
let $F_{\beta}=F_{\beta,0,\xi}$.

We begin with the proof of the $\Gamma$-$\liminf$ inequality.
\begin{lem}
\textbf{\label{lem:(-liminf.)}($\Gamma$-liminf.)} If $(\nu_{\beta})\in\cA$
converges $\nu_{\beta}\to\nu$ weakly, then
\[
\liminf\,F_{\beta}\left(\nu_{\beta}\right)\geq GS\left(\nu\right).
\]
\end{lem}
\begin{proof}
Without loss of generality, assume that $\left\{ \nu_{\beta}\right\} $
satisfies $F_{\beta}\left(\nu_{\beta}\right)\leq C$. This implies
that $\nu_{\beta}\in X_{\beta}$, so that $d\nu_{\beta}=\beta\mu_{\beta}\left[0,t\right]dt$
for some $\mu_{\beta}\in\Pr([0,1])$, and 

\begin{align*}
\liminf\,F_{\beta}\left(\nu_{\beta}\right) & =\liminf\int_{0}^{1}\xi''(s)\beta\hat{\mu}_{\beta}(s)ds+\int_{0}^{1}\frac{1}{\beta}\left(\frac{1}{\hat{\mu}\left(s\right)}-\frac{1}{1-s}\right)\,ds=\liminf\,I_{\beta}+II_{\beta}.
\end{align*}
By properties of weak convergence, $\lim\,\nu_{\beta}[s,1]=\nu[s,1],$
$\cL$- a.e., from which it follows that the first term converges,
\[
\lim\,I_{\beta}=\int_{0}^{1}\xi''\left(s\right)\nu[s,1]ds
\]
by the bounded convergence theorem. Now consider the second term.
By Fatou's lemma, 

\[
\liminf\,II_{\beta}\geq\int_{0}^{1}\liminf\,\frac{1}{\nu_{\beta}[s,1]}\left(1-\frac{1}{\beta}\frac{\nu_{\beta}[s,1]}{1-s}\right)\,ds=\int_{0}^{1}\frac{1}{\nu[s,1]}\,ds.
\]
Therefore,
\[
\liminf\,F_{\beta}\left(\nu_{\beta}\right)\geq GS\left(\nu\right)
\]
as required.\end{proof}
\begin{lem}
\label{lem:(-limsup.-)}($\Gamma$-limsup.) For every $\nu\in\cA$,
there exists a sequence $\left\{ \nu_{\beta}\right\} \subset\cA$
such that $\nu_{\beta}\to\nu$ weakly and 
\[
\limsup\,F_{\beta}\left(\nu_{\beta}\right)\leq GS\left(\nu\right).
\]
\end{lem}
\begin{proof}
Write $d\nu=m\left(t\right)dt+c\delta_{1}$ and define 
\[
c_{\beta}=\begin{cases}
c & c>0\\
\frac{1}{\beta} & c=0
\end{cases}
\]

\begin{claim}
For $\beta>\nu[0,1]+1$, there exists a $q_{\beta}\in(0,1)$ with
the following properties 
\begin{enumerate}
\item $\int_{q_{\beta}}^{1}m(t)dt+c_{\beta}=\beta(1-q_{\beta}).$
\item $q_{\beta}\to1$
\item $\frac{m(q_{\beta})}{\beta}\leq1$
\end{enumerate}
\end{claim}
\begin{proof}
To see (1), let $f(t)=\int_{t}^{1}m(t)dt+c_{\beta}$ and $g(t)=\beta(1-t)$.
Then 
\[
g(0)=\beta>\nu[0,1]+1>\int_{0}^{1}m(t)dt+c_{\beta}=f(0)
\]
and 
\[
0=g(1)<f(1)=c_{\beta}.
\]
The result then follows by the intermediate value theorem. Choose
any such $q_{\beta}$.

To see (2), observe that 
\[
0\leq\beta(1-q_{\beta})\leq\int_{0}^{1}m(t)dt+c_{\beta}\leq C(\nu).
\]
To see (3), observe that by (1), 
\[
\fint_{q_{\beta}}^{1}\frac{m(t)}{\beta}dt\leq1.
\]
 Since $m$ is non-decreasing, (3) follows.
\end{proof}
Let $\beta$ and $q_{\beta}$ be as in the above claim and let 
\[
\mu_{\beta}[0,t]=\begin{cases}
\frac{m(t)}{\beta} & t<q_{\beta}\\
1 & t\geq q_{\beta}
\end{cases}.
\]
Note that this defines a probability measure so that $\mu_{\beta}\in\Pr[0,1].$
This gives us $d\nu_{\beta}=\beta\mu_{\beta}[0,t]dt\in X_{\beta}$

First observe that $\nu_{\beta}\to\nu$. To see this, observe that
if $\phi\in C([0,1])$, 
\begin{align*}
\int\phi d\nu_{\beta} & =\int_{0}^{q_{\beta}}\phi mdt+\int_{q_{\beta}}^{1}\beta\phi(t)dt\\
 & =\int_{0}^{q_{\beta}}\phi mdt+\left(\int_{q_{\beta}}^{1}mdt+c_{\beta}\right)\fint_{q_{\beta}}^{1}\phi(t)\to\int\phi mdt+c\phi(1)
\end{align*}
as desired. Now by definition we have

\[
\limsup\,F_{\beta}\left(\nu_{\beta}\right)=\limsup\,I_{\beta}+II_{\beta},
\]
so it suffices to show that $I_{\beta}\to I$ and $II_{\beta}\to II$.

Now since $\int_{0}^{1}d\nu_{\beta}\to\int_{0}^{1}d\nu$, it follows
that $\xi''(s)\int_{s}^{1}d\nu_{\beta}$ is bounded so that by the
bounded convergence theorem 
\[
\lim I_{\beta}\to\int\xi''(s)\nu[s,1]ds.
\]
Now consider 
\begin{align*}
II_{\beta} & =\int_{0}^{1}\frac{1}{\nu_{\beta}[s,1]}\left(1-\frac{1}{\beta}\frac{\nu_{\beta}[s,1]}{1-s}\right)\,ds=\int_{0}^{q_{\beta}}\frac{1}{\nu_{\beta}[s,1]}\left(1-\frac{1}{\beta}\frac{\nu_{\beta}[s,1]}{1-s}\right)ds.
\end{align*}
On $s\leq q_{\beta}$, we have 
\begin{align*}
\nu_{\beta}[s,1] & =\int_{s}^{q_{\beta}}m(t)dt+\beta(1-q_{\beta})=\int_{s}^{q_{\beta}}m(t)dt+\int_{q_{\beta}}^{1}m(t)dt+c_{\beta}=\int_{s}^{1}m(t)dt+c_{\beta}
\end{align*}
so that 
\[
II_{\beta}=\int_{0}^{q_{\beta}}\frac{1}{\int_{s}^{1}m(t)dt+c_{\beta}}\left(1-\frac{1}{\beta}\frac{\int_{s}^{1}m(t)dt+c_{\beta}}{1-s}\right)ds
\]
Observe that 
\[
0\leq\frac{1}{\int_{s}^{1}m(t)dt+c_{\beta}}\left(1-\frac{1}{\beta}\frac{\int_{s}^{1}m(t)dt+c_{\beta}}{1-s}\right)\indicator{[0,q_{\beta}]}\uparrow\frac{1}{\nu[s,1]}\indicator{[0,1)}
\]
so that by the monotone convergence theorem, 
\[
II_{\beta}\to\int_{0}^{1}\frac{1}{\nu[s,1]}ds
\]
as desired.
\end{proof}

\subsubsection{Proof of \prettyref{thm:gamma-conv-min-conv}}
\begin{proof}
The $\Gamma$-convergence result for $h=0$ and $\xi_{\beta}=\xi$
follows by \prettyref{lem:(-liminf.)} and \prettyref{lem:(-limsup.-)}.
In the case that, $\frac{h^{2}}{\beta^{2}}\to\bar{h}$ and $\xi''_{\beta}\to\xi''$
uniformly, the corresponding result then follows by a continuous perturbation argument (see \prettyref{thm:stability-GC}).

We now turn to the convergence of minimizers and \eqref{eq:GSE-vf}.
First observe that by \prettyref{lem:A-set-cpt}, the family $F_{\beta}$
are sequentially equi-coercive. Since $GS$ has a unique minimizer,
the result follows by \prettyref{thm:FTGC}. 
\end{proof}

\subsection{Moderate Deviations of the minimizers}

In the following we study consequences of the $\Gamma-$convergence.
In particular, we aim to prove \prettyref{cor:moderate-dev}. By \prettyref{thm:gamma-conv-min-conv},
we know that $\nu_{\beta}=\beta\mu_{\beta}[0,t]dt\to d\nu=m(t)dt+c\delta_{1}$
weakly where $d\nu$ is the unique minimizer of GS.
\begin{lem}
\label{lem:cdfs-conv}We have $\beta\mu_{\beta}[0,t]\to m(t)$ for
$t\in[0,1)\cap CtyPts(m(t))$. \end{lem}
\begin{proof}
Let $f_{\beta}=\beta\mu_{\beta}[0,t]$ . Observe first that since
$f_{\beta}dt\to d\nu$ , we have that for every $T\in[0,1)$, 
\[
\sup_{t\in[0,T]}\abs{f_{\beta}(t)}=f_{\beta}(T)\leq\frac{1}{1-T}\int_{T}^{1}f_{\beta}(t)dt=\frac{1}{1-T}\nu_{\beta}\left(\left[T,1\right]\right)\leq C(T).
\]
Thus $f_{\beta}$, restricted to the interval $[0,T]$ is a sequence
of uniformly bounded monotone functions. As a consequence, every subsequence
has a further subsequence that converges point-wise on $[0,T]$ to
some function $f(t)$, and that $f(t)$ also has this bound. 

But then, by the dominated convergence theorem applied to this subsequence,
\[
\int_{0}^{T}g(t)f_{\beta}(t)dt\to\int_{0}^{T}g(t)f(t)dt
\]
for any $g\in L^{1}[0,T]$. As a result, $f(t)=m(t)$ a.e. on $[0,T]$.
By monotonicity, $f(t)=m(t)$ at their continuity points on $[0,T]$.
The subsequence principle applied to $f_{\beta}(t)$ at continuity
points of $m(t)$ on $[0,T]$. Since this holds for each $T<1$ we
conclude the result. \end{proof}
\begin{lem}
\label{lem:atom-conv}Let $q_{\beta}\to1$ be such that $\beta\mu_{\beta}[0,q_{\beta})\to m(1^{-})$
, and suppose that $m(1^{-})<\infty$. If $Y_{\beta}$ have law $\mu_{\beta}$
, then 
\begin{align*}
\nu\left(\left\{ 1\right\} \right) & =\lim\E_{\mu_{\beta}}\left(\beta(1-Y_{\beta})\vert Y_{\beta}\in[q_{\beta},1]\right).
\end{align*}
\end{lem}
\begin{proof}
Note that since $m(1^{-})<\infty$ and $\mu_{\beta}[0,1]=1$, so that
$\mu[q_{\beta},1]\to1$. Thus it suffices to show that 
\[
\lim\beta\int\indicator{[q_{\beta},1]}(1-t)d\mu_{\beta}\to\nu\left(\left\{ 1\right\} \right).
\]
 To see this observe that 
\[
\beta\int_{0}^{1}\mu_{\beta}[0,t]dt=\beta\int_{0}^{q_{\beta}}\mu_{\beta}[0,t]dt+\beta\mu[0,q_{\beta})(1-q_{\beta})+\beta\int_{q_{\beta}}^{1}\mu_{\beta}[q_{\beta},t]dt=I+II+III.
\]
 Observe that $II\to0$ by assumption. Furthermore observe that 
\[
III=\beta\int\indicator{[q_{\beta},1]}(1-t)d\mu_{\beta}.
\]
Since $\nu=m(t)dt+c\delta_{1}$ and $\nu_{\beta}\to\nu$, 
\[
\beta\int_{0}^{1}\mu_{\beta}[0,t]dt\to\int_{0}^{1}m(t)dt+c,
\]
It thus suffices to show that $I\to\int_{0}^{1}m(t)dt$. To see this
observe that if we define $g_{\beta}(t)=\indicator{[0,q_{\beta}]}\beta\mu_{\beta}[0,t]$,
then 
\[
\sup\abs{g_{\beta}(t)}\leq\beta\mu_{\beta}[0,q_{\beta})\leq C
\]
and $g_{\beta}(t)\to m(t)$ for $\cL$-a.e. t by \prettyref{lem:cdfs-conv}.
The result then follows by the dominated convergence theorem.
\end{proof}

\subsubsection{\emph{Proof of \prettyref{cor:moderate-dev}.}}
\begin{proof}
Proof of $(1)$. \prettyref{thm:gamma-conv-min-conv} implies $\beta\mu_{\beta}[0,t]dt\to d\nu$.
A simple integration by parts argument then shows the first result.

Proof of $(2)$. The second result follows from \prettyref{lem:cdfs-conv}.

Proof of (3). The third result comes from \prettyref{lem:atom-conv}. 
\end{proof}

\section{The Dual of the ground state energy Problem\label{sec:The-Dual-of}}

In this section, we prove \prettyref{thm:duality}. We follow the
usual method of introducing an auxiliary function and proving a minmax
theorm for it. The result then follows by studying the min-max and
the max-min problems. We then present a preliminary analysis of the
two functionals which will be important for the regularity theory
of these problems.\textbf{ For the purposes of this section}, we will
think of $\xi$ and $h$ as fixed and write $P=P_{\xi,h}$ when it
is unambiguous.

\subsection{Proof of duality}

In this subsection we prove the duality theorem, \prettyref{thm:duality}.
We introduce the following notation for ease of comparison with \cite{BNS72,NirenbergTopicsNFA01}.
This notation, with the exception of the set $B$, will be used only
for the following lemma. Let $F=C([0,1])$, $G=\cM([0,1])$, $A=\cC$,
and 
\begin{align*}
B & =\left\{ \mu\in\cM_{+}:\mu(\{0,1\})=0\right\} .
\end{align*}
Let $S(\nu):\cM_{+}\to\R$ be defined by 
\[
S(\nu)=2\int\sqrt{\nu_{ac}(x)}dx.
\]
Note that by Jensen's inequality this is finite. Basic regularity
of $S$ is shown in \prettyref{sub:Square-Roots-of}. One important
fact from the latter section that will be used frequently in the subsequent
is the representation 
\begin{equation}
S(\nu)=\inf_{\phi\geq0}\,\int\phi d\nu+\int\frac{1}{\phi}dx.\label{eq:D-inf-rep}
\end{equation}

Define $K:A\times B\to\R$ by 
\[
K\left(u,v\right)=(\xi''-v,u)+S(v)+h^{2}u(0).
\]
Give $F$ the norm topology and $G$ the norm-topology, and give $A,B$
the induced topologies.
\begin{lem}
\label{lem:abstract-duality-kyfan}We have
\[
\inf_{u\in A}\sup_{v\in B}\,K\left(u,v\right)=\sup_{v\in B}\inf_{u\in A}\,K\left(u,v\right).
\]
\end{lem}
\begin{proof}
We use a generalization of Ky Fan's min-max theorem due to Brezis,
Nirenberg, Stampacchia \cite{BNS72,NirenbergTopicsNFA01}. Note that
$A,B$ are convex sets, and that $F$ is a Hausdorff topological vector
space. We need to check
\begin{enumerate}
\item For each $v\in B$, $K$ is quasi-convex in $u$ and l.s.c. on $A$
\item For each $u\in A$, $K$ is quasi-concave in $v$ and u.s.c. on $B$
\item For some $\tilde{v}\in B$, and some $\lambda>\sup_{B}\inf_{A}K$,
the set $\left\{ u\in A\ :\ K\left(u,\tilde{v}\right)\leq\lambda\right\} $
is compact.
\end{enumerate}
Then the Generalized Ky Fan min-max theorem will imply the result.

Pf of 1. Let $v\in B$, then the map 
\[
A\to\R,\ u\to K\left(u,v\right)
\]
is affine with continuous linear part, hence it is (quasi-)convex
and continuous.

Pf of 2. Let $u\in A$, then the map 
\[
B\to\R,\ v\to K\left(u,v\right)
\]
is concave by concavity of $x\mapsto\sqrt{x}$. It is upper semi-continuous
by \prettyref{cor:sqrt-usc}.

Pf of 3. We begin by asserting that $\sup_{B}\inf_{A}K<\infty$. To
see that, let $P$ is as in \prettyref{eq:Pfunc}. By \eqref{eq:D-inf-rep},
it follows for every $u\in A$, 
\begin{equation}
\sup_{v\in B}K(u,v)\leq P(u)\label{eq:K-P-upperbound}
\end{equation}
so that 
\[
\sup_{B}\inf_{A}K\leq\inf_{A}\sup_{B}K\leq\inf_{A}P\leq P(1)<\infty
\]
Therefore there exists $\lambda\in\R$ such that $\lambda>\sup_{B}\inf_{A}K$.
Since the set 
\[
E=\left\{ u\in A\ :\ K\left(u,0\right)\leq\lambda\right\} 
\]
is a closed subset of the set 
\[
\{u\in A:0\leq(\xi'',u)\leq\lambda\},
\]
\prettyref{lem:C-set-cpt} in the case $f=\xi''$, yields that $E$
is norm compact in $C[0,1]$.\end{proof}
\begin{thm}
\label{thm:P-D-duality-proof}We have
\[
\min_{\phi\in\cC}\,P\left(\phi\right)=\sup_{\eta\in\cK_{h,\xi}}\,D\left(\eta\right).
\]
\end{thm}
\begin{proof}
Our first goal is to show that for all $\phi\in A$, 
\begin{align}
\sup_{\sigma\in B}K\left(\phi,\sigma\right) & =P\left(\phi\right).\label{eq:supK-is-P}
\end{align}
The upper bound is given by \prettyref{eq:K-P-upperbound}. On the
other hand, for all $\epsilon>0$ we let $\sigma_{\epsilon}(dx)=\frac{1}{(\phi+\epsilon)^{2}}dx$,
to get that 
\begin{align*}
\inf_{\sigma\in B}\left(\left(\phi,\sigma\right)-\int2\sqrt{\sigma_{ac}(x)}dx\right) & \leq\int\frac{\phi}{(\phi+\epsilon)^{2}}dx-2\int\frac{1}{(\phi+\epsilon)}dx\\
 & \leq-\int\frac{1}{\phi+\epsilon}dx\to-\int\frac{1}{\phi}dx
\end{align*}
 as $\epsilon\to0$ by the monotone convergence theorem. Thus \eqref{eq:supK-is-P}
holds.

Next we show that for all $\sigma\in B$,

\[
\inf_{\phi\in\cC}K\left(\phi,\sigma\right)=S(\sigma)+\Xi\left(\sigma\right)
\]
where
\[
\Xi\left(\sigma\right)=\inf_{\phi\in\cC}\left((\phi,\xi'')-(\phi,\sigma)+h^{2}\phi(0)\right)=\begin{cases}
0 & \sigma=\eta'',\ \eta\in\cK_{h,\xi}\\
-\infty & \text{otherwise}
\end{cases}.
\]
The first equality is self-evident. The issue is to show the second
equality. 

Suppose first that there is no $\eta\in K_{\xi}$ such that $\eta''=\sigma$.
Define $\eta$ to be the solution of 
\[
\begin{cases}
\eta''=\sigma\\
\eta'(0)=\xi'(0)-h^{2} & ,\\
\eta(1)=\xi(1)
\end{cases}
\]
(see \prettyref{lem:meas-BVP-unique}). Observe that $\eta$ is continuous,
convex, and has the boundary data from the definition of $\cK_{h,\xi}$.
Thus it must be that $\{\xi>\eta\}$ is non-empty.

Take $L\in\R_{+}$ and $\phi$ satisfying $\phi'(0)=0$, $\phi(1)=0$,
and $-\phi''=L\mu$ where $\mu$ is a probability measure supported
in a compact subset of $\{\xi>\eta\}\backslash\{0,1\}$. Since $\phi\in\cC\cap\left\{ \phi'\in BV\right\} $
it follows by \prettyref{lem:(Integration-by-parts)} that 
\begin{align*}
\left(\left(\xi-\eta\right)'',\phi\right)+h^{2}\phi(0) & =\left(\xi-\eta\right)'\phi\vert_{0}^{1}-\left(\xi-\eta\right)\phi'\vert_{0}^{1}+\left(\eta-\xi,-\phi''\right)+h^{2}\phi(0)\\
 & =L\left(\eta-\xi,\mu\right).
\end{align*}
Taking $L\to\infty$ gives the result. 

We now need to show that the infimum is $0$ when $\sigma=\eta''$
for some $\eta\in\cK_{h,\xi}$. To see this note that by definition
of $\cK_{h,\xi}$, if $\phi'(1)>-\infty$, the same integration by
parts argument yields 
\begin{align*}
\left(\left(\xi-\eta\right)'',\phi\right)+h^{2}\phi(0) & =(\xi-\eta)'(1)\phi(1)+\left(\xi-\eta\right)(0)\phi'(0)+\left(\eta-\xi,-\phi''\right)\geq0.
\end{align*}
We used here that $\eta\in\cK_{h,\xi}$ then $(\xi-\eta)'(1)\geq0$
by \prettyref{lem:K-xi-bv}. Now if $\phi'(1)=-\infty$, we take a
sequence of $\phi_{n}\in\cC$ with $\phi_{n}\to\phi$ in norm with
$\phi'_{n}(1)>-\infty$ (see \prettyref{lem:C-finite-deriv-approx}),
for which the inequality still holds and then pass to the limit.

The duality then follows by \prettyref{lem:abstract-duality-kyfan}
which implies that
\[
\inf_{\phi\in\cC}P(\phi)=\inf_{\phi\in A}\sup_{\sigma\in B}\,K\left(\phi,\sigma\right)=\sup_{\sigma\in B}\inf_{\phi\in A}\,K\left(\phi,\sigma\right)=\sup_{\eta\in\cK_{\xi}}D(\eta'').
\]

\end{proof}

\subsection{Preliminary Analysis of the Primal-Dual Relationship}

In this section we do some preliminary analyses of the Primal and
Dual problems and their relationship which will be used in the subsequent.
\begin{lem}
\label{lem:dual_optimal}$(\phi,\eta)\in\cC\times\cK_{h,\xi}$. The
following are equivalent:
\begin{itemize}
\item $P\left(\phi\right)=D\left(\eta\right)$
\item $\phi^{2}\eta''\left(dx\right)=dx$ and $\int\phi\xi''+h^{2}\phi(0)=\int\phi\eta''$
\item $\phi$ and $\eta$ optimize $P$ and $D$ respectively.
\end{itemize}
Furthermore, if $P(\phi)=D(\eta)$, we have that $\frac{1}{\phi^{2}}\in L^{1}$,
$\eta''<<dx$, $\eta''(dx)=\frac{1}{\phi^{2}}dx$, and 
\begin{align*}
\xi'(1)=\eta'(1) & \text{ or }\phi(1)=0\\
\xi(0)=\eta(0) & \text{ or }\phi'(0)=0.
\end{align*}
\end{lem}
\begin{proof}
By the same argument as in \prettyref{thm:P-D-duality-proof}, for
any such pair we have the inequality 
\begin{align*}
I=(\phi,\xi'')-(\phi,\eta'') & +h^{2}\phi(0)\geq0.
\end{align*}
In fact, any such pair must satisfy 
\[
II=\int\frac{1}{\phi}+(\phi,\eta'')-S(\eta'')\geq0.
\]
To see this, observe that 
\[
II=\int\frac{1}{\phi(x)}+\phi(x)\eta''_{ac}(x)-2\sqrt{\eta''_{ac}(x)}dx+\int\phi\eta''_{sing}(dx)
\]
which is non-negative as both integrands are non-negative. Hence,
\begin{equation}
P(\phi)-D(\eta)=I+II\geq0\label{eq:PaboveD}
\end{equation}
with the case of equality if and only if $I=II=0$.

From the statement that $I=0$, we conclude by an integration by parts
and approximation argument (as in the proof of \prettyref{thm:P-D-duality-proof})
that
\begin{align*}
\xi'(1)=\eta'(1) & \text{ or }\phi(1)=0\\
\xi(0)=\eta(0) & \text{ or }\phi'(0)=0.
\end{align*}
Furthermore, we have that $II=0$ if and only if $\phi^{2}\eta''(dx)=dx$.
Indeed, if $II=0$, then $\phi^{2}(x)\eta''_{ac}(x)=1$ $\cL$-a.e.\
and $\text{supp}\,\eta''_{sing}\subset\left\{ \phi=0\right\} $. Since
$\frac{1}{\phi}\in L^{1}$, it follows from monotonicity of $\phi$
that $\eta''_{sing}(dx)=0$. Hence $\phi^{2}\eta''(dx)=dx$ and $\eta''_{ac}=\frac{1}{\phi^{2}}\in L^{1}$.
The reverse direction is clear.

That $P(\phi)=D(\eta)$ if and only if $\phi$ and $\eta$ are optimal
is an immediate consequence of \eqref{eq:PaboveD}.\end{proof}
\begin{lem}
\label{lem:primal-variation}Let $\phi\in\cC$ then for any $\psi\in\cC$,
\[
\frac{d}{d\tau}\vert_{\tau=0^{+}}P(\phi+\tau\psi)=\left(\xi''-\frac{1}{\phi^{2}},\psi\right)+h^{2}\psi(0)
\]
Furthermore, $\phi$ is optimal only if 
\[
(\xi''-\frac{1}{\phi^{2}},\psi)+h^{2}\psi(0)\geq0.
\]
\end{lem}
\begin{proof}
Notice that it suffices to show that the nonlinear term is right differentiable.
Since $\psi\in\cC$, so is $\phi_{\tau}=\phi+\tau\psi$. Now $\mbox{\ensuremath{\psi\geq}0 }$
so that $1/\phi_{\tau}$ is a non-negative, monotone increasing sequence
of functions. then 
\[
\lim_{\tau\to0^{+}}-\frac{1}{\tau}\int\frac{1}{\phi_{\tau}}-\frac{1}{\phi}dx=\lim_{\tau\to0^{+}}\int\frac{\psi}{\phi_{\tau}\phi}dx=\int\frac{1}{\phi^{2}}\psi dx
\]
by the monotone convergence theorem so that the non-linear term in
$P$ is right differentiable at $\tau=0$. The second claim follows
from first order optimality.\end{proof}
\begin{lem}
\label{lem:phi-inverse-integrable}If $\phi$ optimizes $P(\cdot)$,
then $\frac{1}{\phi^{2}}\in L_{1}.$ In particular, 
\[
\eta(x)=\xi(1)-\int_{x}^{1}\int_{0}^{y}\frac{1}{\phi^{2}}dzdy+h^{2}(1-t)
\]
is in $\cK_{h,\xi}\cap W^{2,1}$.\end{lem}
\begin{proof}
Fix $\nu\in\cM_{+}$ and let $\psi(t)=a+c(1-t)+\int_{t}^{1}\nu[0,s]dt$
then $\psi\in\cC$ so that by \prettyref{lem:primal-variation}, 
\[
\frac{d}{d\tau}\vert_{\tau=0}P(\phi+\tau\psi)=\int(\xi''-\frac{1}{\phi^{2}})\psi dt+h^{2}a\geq0
\]
Choosing $a=1,c=0,$ and $\nu=0$ gives 
\[
\int\xi''-\frac{1}{\phi^{2}}dt+h^{2}\geq0
\]
so that 
\[
h^{2}+\int\xi''\geq\int\frac{1}{\phi^{2}}\geq0.
\]
 (To avoid adding infinities, subtract $h^{2}+\int\xi''$ from both
sides and use the a priori sign on $-\frac{1}{\phi^{2}}$.) 

$\eta(x)$ is plainly continuous, convex, and has the correct boundary
data. It suffices to show that $\eta\geq\xi$. Observe that for any
$\psi\in\cC$ with $\psi'(1)>-\infty$, we have 
\[
0\leq\left(\psi,\xi''-\frac{1}{\phi^{2}}\right)+h^{2}\psi(0)=\psi(1)(\xi'-\eta')(1)+\psi'(0)(\xi-\eta)(0)+\left(\psi'',\xi-\eta\right)
\]
by \prettyref{lem:primal-variation}. Taking $\psi$ with $\psi(1)=0,\psi'(0)=0$
and $\psi''=-\delta_{t}$ for $t\in(0,1)$ shows that $\eta(t)\geq\xi(t).$
Then inequality then extends by continuity. Thus $\eta\in\cK_{h,\xi}$.
That $\eta\in W^{2,1}$ is immediate. \end{proof}
\begin{thm}
\label{thm:(Well-posedness)}(Well-posedness) The Dual problem is
well-posed. In particular, 
\[
\sup_{\eta\in\cK_{h,\xi}}D(\eta)=\max_{\eta\in\cK_{h,\xi}}\,D\left(\eta\right).
\]
Furthermore, $\eta$ is unique.\end{thm}
\begin{proof}
This from the fact that $P$ has a unique minimizer combined with
\prettyref{lem:dual_optimal} and \prettyref{lem:phi-inverse-integrable}.
\end{proof}

\section{Regularity theory for the optimizers\label{sec:Regularity-theory-for} }

In this section we prove \prettyref{thm:(a-priori-estimates)} and
\prettyref{thm:wiggles-principle}. In this section, $\phi$ and $\eta$
refer \textbf{exclusively} to the optimizers. (That there exists a
unique optimizer $\eta$ was proved in the previous section.)

\subsection{Primal and Dual Regularity}

Our goal in this section is to prove \prettyref{thm:(a-priori-estimates)}.
We will prove this by first getting weaker regularity using the above
and then we will upgrade this regularity. 

Notice that by \prettyref{lem:dual_optimal}, we have the following
lemma.
\begin{lem}
\label{lem:dual-reg}We have that
\begin{enumerate}
\item $\phi>0$ on $[0,1)$
\item $\eta\in C_{loc}^{2}\left((0,1)\right)\cap C^{1}([0,1])$.
\item $\frac{1}{\sqrt{\eta''}}=\phi$ on $(0,1)$
\item on the set $\left\{ \xi=\eta\right\} \cap(0,1)$, $\eta'(t)=\xi'(t)$,
and $\eta''\left(t\right)\geq\xi''\left(t\right).$
\end{enumerate}
\end{lem}
\begin{proof}
Begin by observing that $(1)-(3)$ follow from \prettyref{lem:dual_optimal}
and monotonicity of $\phi$.

To see $(4)$, fix $t\in(0,1)$ a contact point. Begin by recalling
that \textbf{$\eta$ }is left/right differentiable and lies above
$\xi$ by definition of $\cK$, so $\xi'(t)$ is in the sub-differential
\[
\partial_{-}\eta(t)\leq\xi'(t)\leq\partial_{+}\eta(t)
\]
but $\eta\in C_{loc}^{2}((0,1))$ so these are in fact equalities.
Furthermore if we define $g=\eta-\xi\geq0$, then for $\epsilon$
sufficiently small, 
\[
g\left(t+\epsilon\right)=g\left(t\right)+g'\left(t\right)\epsilon+\frac{g''\left(\tau\right)}{2}\epsilon^{2}=\frac{g''\left(\tau\right)}{2}\epsilon^{2}
\]
for some $\tau$ between $t$ and $t+\epsilon$ where the second equality
follows by the observations follows by the argument earlier in the
lemma. Since $g\in C_{loc}^{2}((0,1))$ and $g\geq0$ we conclude
that $g''\left(t\right)\geq0$ by taking $\epsilon\to0$.\end{proof}
\begin{lem}
\label{lem:eta-dist-diff} We have that $\left(\frac{1}{\sqrt{\eta''}}\right)''=-\mu$
in the sense of $\cD'$ where $\mu$ is a positive Radon measure on
$(0,1)$ which has support in $\mbox{\{\ensuremath{\xi}=\ensuremath{\eta}\}}$. \end{lem}
\begin{proof}
By \prettyref{lem:dual-reg}, $\frac{1}{\sqrt{\eta''}}$ is concave
so that the statement 
\[
\left(\frac{1}{\sqrt{\eta''}}\right)''=-\mu
\]
follows from Alexandrov's theorem (\prettyref{thm:(Modified-Alexandrov)}).
To find the condition on the support, it suffices to consider the
case that the contact set is not the whole interval. Then there is
an $\epsilon_{0}>0$ and an open set $U$ with $U\subset\subset\{\eta-\xi>\epsilon_{0}\}$.
Let $\sigma$ be a bump function that is localized in $U$ . Note
that we can take $U$ to be away from $\{0,1\}$. Now $\eta\in C_{loc}^{2}(0,1)$
by \prettyref{lem:dual-reg} so that it is $C^{2}(\bar{U})$. Furthermore
$\eta''(x)\geq\inf_{U}\frac{1}{\phi^{2}}>0$. Thus one can directly
verify that $\eta+\tau\sigma\in K_{\xi}$ for $\abs{\tau}$ sufficiently
small. This implies 
\[
\frac{d}{d\tau}\vert_{\tau=0}S(\eta+\tau\sigma)=\left(\frac{1}{\sqrt{\eta''}},\sigma''\right)=0
\]
by the optimality of $\eta$ and the fact that this is a full derivative
(i.e. not just a right or left derivative). The result then follows
by the definition of distributional derivatives.\end{proof}
\begin{lem}
\label{lem:mass-bounds}There is a $c>0$ with $0<c\leq\phi.$ In
particular $\eta\in C^{2}([0,1])$, and the second derivative has
the strict lower bound $\eta''(x)>c^{-2}$.\end{lem}
\begin{proof}
We do a case analysis. Suppose there exists a sequence of interior
contact points $\left\{ t_{n}\right\} $ with $t_{n}\uparrow1$. Then
by \prettyref{lem:dual-reg}, you have that $\eta'\left(t_{n}\right)=\xi'\left(t_{n}\right)$
for all $n$, so then by mean value theorem there exists a sequence
of points $\left\{ \tilde{t}_{n}\right\} $ with $\tilde{t}_{n}\uparrow1$
at which $\eta''\left(\tilde{t}_{n}\right)=\xi''\left(\tilde{t}_{n}\right)$.
Then by conjugacy you have that $\phi^{2}\left(\tilde{t}_{n}\right)=\frac{1}{\sqrt{\xi''\left(\tilde{t}_{n}\right)}}$
and therefore $\phi^{2}\left(1\right)=\frac{1}{\sqrt{\xi''\left(1\right)}}$
which gives a lower bound on $\phi$. 

On the other hand, if no such sequence of contact points exists, then
there is an interval of the form $\left(a,1\right)\subset\left\{ \xi<\eta\right\} $.
By \prettyref{lem:eta-dist-diff}, $\left(\frac{1}{\sqrt{\eta''}}\right)''=0$
in the sense of distributions on $\left(a,1\right)$. By \prettyref{lem:dual-reg}
and elementary properties of distributions \cite{LiebLoss2001}, it
follows that $\phi$ is linear on $\left(a,1\right)$. Combined with
finite energy you conclude that $\phi\left(1\right)\neq0$. This gives
a lower bound on $\phi$ by monotonicity. That $\eta\in C^{2}$ then
follows by a continuity argument.\end{proof}
\begin{lem}
\label{lem:m-finite} We have that $\phi'(1)>-\infty$. Thus, $\phi\in C([0,1])\cap\{\phi'\in BV((0,1))\}$.
For the corresponding $d\nu=m(t)dt+c\delta_{1}\in\cA$, we have that
$m(1-)<\infty$. \end{lem}
\begin{proof}
Note that by \prettyref{lem:dual-reg} and \prettyref{lem:mass-bounds},
we have that $\eta''(x)=\frac{1}{\phi^{2}(x)}$ for all $x\in[0,1]$.
Thus, $\eta''$ is left-differentiable at $1$ (with possibly infinite
value) and in particular, 
\begin{align*}
\eta'''(1) & =\lim_{t\to1^{-}}\frac{\eta''(1)-\eta''(t)}{1-t}=\lim_{t\to1^{-}}\frac{\frac{1}{\phi^{2}(1)}-\frac{1}{\phi^{2}(t)}}{1-t}\\
 & =\lim_{t\to1^{-}}\frac{\phi^{2}(t)-\phi^{2}(1)}{(1-t)\phi^{2}(1)\phi^{2}(t)}=\lim_{t\to1^{-}}\frac{(\phi(t)-\phi(1))}{(1-t)}\frac{(\phi(t)+\phi(1))}{\phi^{2}(1)\phi^{2}(t)}\\
 & =-\phi'(1)\frac{2}{\phi^{3}(1)}.
\end{align*}
Since $\phi^{3}(1)\in(0,\infty)$, we see that the first claim is
equivalent to proving that $\eta'''(1)<\infty$. 

Now we proceed by a case analysis. Suppose first that $1$ is an isolated
contact point. Then as in the proof of \prettyref{lem:mass-bounds}
we conclude that $\phi$ is linear on an interval of the form $(1-\delta,1)$.
It follows immediately that $\phi'(1)>-\infty$.

Suppose now that $1$ is not an isolated contact point. Then there
is a monotone sequence $t_{i}\in(0,1)$ such that $t_{i}\uparrow1$
and $\eta(t_{i})=\xi(t_{i})$. By \prettyref{lem:dual-reg}, we have
that $\eta'(t_{i})=\xi'(t_{i})$ and that $\eta''(t_{i})\geq\xi''(t_{i})$.
This easily implies that $\eta'''(1)=\xi'''(1)<\infty$. To see this,
observe that by the mean value theorem there is a sequence $\tau_{i}\in(t_{i},t_{i+1})$
such that $\eta''(\tau_{i})=\xi''(\tau_{i})$. Then, $\eta''(1)=\xi''(1)$
since $\eta\in C^{2}([0,1])$. Hence,
\begin{align*}
\eta'''(1) & =\lim_{t\to1}\frac{\eta''(1)-\eta''(t)}{1-t}=\lim_{i\to\infty}\frac{\eta''(1)-\eta''(\tau_{i})}{1-\tau_{i}}\\
 & =\lim_{i\to\infty}\frac{\xi''(1)-\xi''(\tau_{i})}{1-\tau_{i}}=\xi'''(1)
\end{align*}
as required.

Having shown that $\phi'(1)>-\infty$, we immediately conclude that
$\phi'\in BV$ (see \prettyref{lem:phi'finite}). The last claim follows
from the correspondence between $\cA$ and $\cC$, which gives that
$\phi(t)=\int_{t}^{1}m(t)+c$ for all $t\in[0,1]$ and hence that
$m(1^{-})=-\phi'(1)<\infty$.
\end{proof}

\subsubsection{Proof of \prettyref{thm:(a-priori-estimates)}}
\begin{proof}
This follows from applying Lemmas \ref{lem:dual_optimal}, \ref{lem:dual-reg}, \ref{lem:eta-dist-diff}, \ref{lem:mass-bounds}, and
\ref{lem:m-finite}.
\end{proof}

\subsection{Regularity of the contact set}

In this section we prove \prettyref{thm:wiggles-principle}. This result, which we restate below for the convenience of the reader, provides
a simple test which characterizes which sub-intervals of $[0,1]$
are compatible with FRSB and which are compatible with $kRSB$. 
\begin{thm*}\textbf{\ref{thm:wiggles-principle}.}
Let 
\[
\mathfrak{\mathfrak{d}}(t)=\left(\frac{1}{\sqrt{\xi''}}\right)''(t)
\]
Then we have the following two cases
\begin{enumerate}
\item If $\mathfrak{d}(t)>0$ on $\left(a,b\right)\subset[0,1]$, then there
are at most two contact points in $\left[a,b\right]$.
\item If $\mathfrak{d}(t)\leq0$ on $[a,b]\subset(0,1]$, then if there
are two contact points $t_{1},t_{2}\in\left[a,b\right]$, then $\left[t_{1},t_{2}\right]\subset\left\{ \eta=\xi\right\} $.
\end{enumerate}
\end{thm*}
\begin{proof}
We begin with $1$. Let $\eta$ be optimal for the dual problem, and
suppose that $\mathfrak{d}(t)>0$ on $\left[a,b\right]$. By contradiction,
assume there are three contact points at $a\leq t_{1}<t_{2}<t_{3}\leq b$.
Note that by \prettyref{thm:(a-priori-estimates)}, $\eta'\left(t_{i}\right)=\xi'\left(t_{i}\right)$
for $i=1,2,3$, and $\eta''\left(t_{2}\right)\geq\xi''\left(t_{2}\right)$.
By the mean value theorem, there exist points $s_{1},s_{2}$ with
$t_{1}<s_{1}<t_{2}<s_{2}<t_{3}$ such that $\eta''\left(s_{i}\right)=\xi''\left(s_{i}\right)$
for $i=1,2$. By \prettyref{thm:(a-priori-estimates)}, $t\to\frac{1}{\sqrt{\eta''\left(t\right)}}$
is concave, and hence the function $t\to\frac{1}{\sqrt{\eta''\left(t\right)}}-\frac{1}{\sqrt{\xi''\left(t\right)}}$
is strictly concave on $(a,b)$ by the assumption on $\mathfrak{d}$.
As this function vanishes at the points $s_{1},s_{2}$, it must be
strictly positive between. In particular, you have that $\xi''\left(t_{2}\right)>\eta''\left(t_{2}\right)$,
which is a contradiction.

Now we turn to $2$. Suppose now that $\mathfrak{d}\leq0$ on $[a,b]\subset(0,1].$
We claim that
\[
\int_{t_{1}}^{t_{2}}\left(\frac{1}{\sqrt{\xi''}}-\frac{1}{\sqrt{\eta''}}\right)\left(\xi-\eta\right)''\geq0.
\]
With this claim in hand, observe that
\[
\frac{1}{\sqrt{\xi''}}-\frac{1}{\sqrt{\eta''}}=\frac{1}{\sqrt{\xi''\cdot\eta''}(\sqrt{\eta''}+\sqrt{\xi''})}\left(\eta''-\xi''\right)
\]
so that the claim implies that 
\[
\int_{t_{1}}^{t_{2}}\frac{1}{\sqrt{\xi''\cdot\eta''}(\sqrt{\eta''}+\sqrt{\xi''})}\left(\xi''-\eta''\right)^{2}\leq0.
\]
The integrand, however, is nonnegative by the same argument. This
implies that in fact
\[
\frac{1}{\sqrt{\xi''\cdot\eta''}(\sqrt{\eta''}+\sqrt{\xi''})}\left(\xi''-\eta''\right)^{2}=0
\]
for $\cL$-a.e. point in $\left[t_{1},t_{2}\right]$, and hence $\eta''=\xi''$
for $\cL$-a.e. point in $\left[t_{1},t_{2}\right]$. Since $\eta=\xi$
at $t_{1},t_{2}$ we conclude the result.

Now we prove the claim. First, note that by the assumption on $\mathfrak{d}$,
\[
\left(\frac{1}{\sqrt{\xi''}}\right)''\left(\eta-\xi\right)\leq0\quad t\in\left[t_{1},t_{2}\right]
\]
since $\eta\in\cK_{h,\xi}$, so that
\[
\int_{t_{1}}^{t_{2}}\left(\frac{1}{\sqrt{\xi''}}\right)''\left(\eta-\xi\right)\leq0.
\]
Integrate by parts and use the boundary conditions given by contact
at $t_{1},t_{2}$ to find that
\[
\int_{t_{1}}^{t_{2}}\frac{1}{\sqrt{\xi''}}\left(\eta-\xi\right)''\leq0.
\]
On the other hand, define 
\[
\tilde{\eta}\left(t\right)=\begin{cases}
\eta\left(t\right) & t\notin\left[t_{1},t_{2}\right]\\
\xi\left(t\right) & t\in\left[t_{1},t_{2}\right]
\end{cases}
\]
and note that $\tilde{\eta}\in\cK_{h,\xi}$. Furthermore, integration
by parts and \prettyref{thm:(a-priori-estimates)}, yield $\tilde{\eta}''\in L^{1}$
with the obvious expression. 

A first variation calculation and optimality of $\eta$ then yields
the first order optimality condition
\[
\int_{t_{1}}^{t_{2}}\frac{1}{\sqrt{\eta''}}\left(\xi-\eta\right)''\leq0.
\]
Subtracting this from the above gives the claim that
\[
\int_{t_{1}}^{t_{2}}\left(\frac{1}{\sqrt{\xi''}}-\frac{1}{\sqrt{\eta''}}\right)\left(\xi-\eta\right)''\geq0.
\]

\end{proof}

\section{1RSB\label{sec:1RSB}}

In this section, we will study the special case that the minimizer
is 1RSB. We begin first with a study of this in the abstract. We remind
the reader of the terminology from \prettyref{sub:Characterization-of-1RSB},
and in particular the modified definition of $RSB_{1}$ given by \prettyref{eq:RSB1}.

\subsection{1RSB}

In this subsection, we will prove \prettyref{cor:consistency1RSB}.
Recall the notion of formal conjugate from \prettyref{sub:Characterization-of-1RSB}.
Observe that the $\eta$ given by \eqref{eq:formal-conj-1} in that
section is continuous, convex, and satisfies the correct boundary
data for $\cK_{\xi}$. However, $\eta$ does not necessarily satisfy
the obstacle condition. 

Regarding the natural boundary conditions in \prettyref{thm:(a-priori-estimates)},
we have the following result whose proof is a straightforward calculation.
\begin{fact}
\label{fact:m,c}Let $\phi\in RSB_{1}$ and let $R$ be given by \eqref{eq:R}.
Then we have that
\begin{align*}
R\left(t\right) & =\frac{1}{m}\left(\frac{1}{m}\log\left(\frac{c+m\left(1-t\right)}{c}\right)-\frac{1-t}{c+m}\right)\\
R'\left(t\right) & =\frac{1}{m}\left(\frac{1}{c+m}-\frac{1}{c+m\left(1-t\right)}\right).
\end{align*}
Furthermore, the formal conjugate $\eta$ to $\phi$ satisfies
\[
\begin{cases}
\eta(0)=\xi(0)\\
\eta'(1)=\xi'(1)
\end{cases}
\]
if and only if $m$ and $c$ solve
\begin{equation}
\begin{cases}
\xi\left(1\right)=\frac{1}{m}\left(\frac{1}{m}\log\left(\frac{c+m}{c}\right)-\frac{1}{c+m}\right)\\
\frac{1}{\xi'\left(1\right)}=c\left(c+m\right)
\end{cases}.\label{eq:FP1RSB}
\end{equation}

\end{fact}
The system \eqref{eq:FP1RSB} can be simplified by eliminating $m$
for $c$. That one can solve this system is proved in \prettyref{fact:yxi}.
\begin{lem}
\label{lem:master-equation} Let $m,c>0$ solve \eqref{eq:FP1RSB}.
Then $c^{2}\xi'(1)<1$, and $c$ solves the master equation 
\begin{equation}
\frac{\xi(1)}{\xi'(1)}=\frac{c^{2}\xi'(1)}{1-c^{2}\xi'(1)}\left(\frac{1}{1-c^{2}\xi'(1)}\log\left(\frac{1}{c^{2}\xi'(1)}\right)-1\right).\label{eq:master-equation}
\end{equation}
\end{lem}
\begin{proof}
Since $m,c$ are positive, we have that
\[
\frac{1}{\xi'(1)}=c^{2}(1+\frac{m}{c})>c^{2}.
\]
Showing that $c$ solves the master equation is a manipulation. Rewrite
the first equation as 
\[
c\left(\frac{1}{c^{2}\xi'(1)}-1\right)=m,
\]
then plug this into \eqref{eq:FP1RSB} to find that 
\begin{align*}
\xi(1) & =\left(\frac{1}{c\left(\frac{1}{c^{2}\xi'(1)}-1\right)}\right)\left(\frac{1}{c\left(\frac{1}{c^{2}\xi'(1)}-1\right)}\log\left(\frac{1}{c^{2}\xi'(1)}\right)-c\xi'(1)\right)\\
 & =\left(\frac{c^{2}\xi'(1)}{1-c^{2}\xi'(1)}\right)\xi'(1)\left[\left(\frac{1}{1-c^{2}\xi'(1)}\right)\log\left(\frac{1}{c^{2}\xi'(1)}\right)-1\right].
\end{align*}

\end{proof}
The next result is regarding a sufficient condition for optimality
and will be used in the subsequent. 
\begin{lem}
\label{lem:suffoptimality1RSB}Assume that $m,c>0$ solve \eqref{eq:FP1RSB}.
Let $\phi=m(1-t)+c$ and let $\eta$ be its formal conjugate. If $\eta\geq\xi$
on $[0,1]$, then $\phi$ and $\eta$ are optimal for $P$ and $D$
.\end{lem}
\begin{proof}
Since $\eta$ satisfies \eqref{eq:formal-conj-1}, we have that $\eta$
is continuous, convex, and satisfies the boundary conditions in $\cK_{\xi}$.
We also have that $\phi^{2}\eta''=1$. By \prettyref{fact:m,c} and
since $\phi''=0$, we have that $\int\xi''\phi=\int\eta''\phi.$ Thus,
the result follows by \prettyref{lem:dual_optimal}.
\end{proof}

\subsubsection{Proof of \prettyref{cor:consistency1RSB}.}
\begin{proof}
Assume that the optimal $\phi$ is in $RSB_{1}$, and write $\phi(t)=m_{*}(1-t)+c_{*}$.
Let $\eta\in\cK_{\xi}$ be the corresponding optimizer. By \prettyref{thm:duality},
\prettyref{thm:(a-priori-estimates)}, and \ref{lem:meas-BVP-unique},
$\eta$ is the formal conjugate to $\phi$ given by \prettyref{eq:formal-conj-1}.
Since $\phi'(0)\neq0$, the natural boundary conditions from \prettyref{thm:(a-priori-estimates)}(d)
imply that $\eta(0)=\xi(0)$ and $\eta'(1)=\xi'(1)$. Now claim (1)
follows from \prettyref{fact:m,c}. Claims (2) and (3) follow from
\prettyref{thm:(a-priori-estimates)}$(c)$ after observing that $0$
and $1$ are contact points.

The following lemma justifies the appearance of the ``pure-like or
critical'' condition appearing in the statement of \prettyref{cor:consistency1RSB}.\end{proof}
\begin{lem}
\label{lem:purelikeorcrit}Assume that $\xi\neq\xi_{SK}$. Let $m,c>0$
solve \eqref{eq:FP1RSB} and let $\eta$ be the formal conjugate to
$\phi(t)=m(1-t)+c$. Then $\eta''\left(1\right)\geq\xi''\left(1\right)$
if and only if $\xi$ is pure-like or critical.\end{lem}
\begin{proof}
First we note that since $\xi\neq\xi_{SK}$, $\xi''(1)>\xi'(1)$.
Now by \eqref{eq:formal-conj-1}, $\eta(t)=\xi(1)-R(t)$ so that 
\[
\eta''(t)=-R''(t)=\frac{1}{\left(c+m\left(1-t\right)\right)^{2}}.
\]
Thus, $\eta''(1)\geq\xi''(1)$ if and only if $c^{2}\leq\frac{1}{\xi''(1)}$.
Let $y=c^{2}\xi'(1)$, then by the master equation\eqref{eq:master-equation},
\[
\xi(1)=\frac{y}{1-y}\xi'(1)\left(\frac{1}{1-y}\log\left(\frac{1}{y}\right)-1\right).
\]
Since the function 
\[
f(s)=\frac{s}{1-s}\xi'(1)\left(\frac{1}{1-s}\log\left(\frac{1}{s}\right)-1\right)
\]
is increasing on $[0,1)$, $c^{2}\leq\frac{1}{\xi''(1)}$ if and only
if $f(y)\leq f(\frac{\xi'(1)}{\xi''(1)})$. Rewriting the last condition
gives 
\begin{align*}
\xi(1) & \leq\frac{\frac{\xi'(1)}{\xi''(1)}}{1-\frac{\xi'(1)}{\xi''(1)}}\xi'(1)\left(\frac{1}{1-\frac{\xi'(1)}{\xi''(1)}}\log\left(\frac{1}{\frac{\xi'(1)}{\xi''(1)}}\right)-1\right)\\
 & =\frac{\xi'(1)^{2}\xi''(1)}{\left(\xi''(1)-\xi'(1)\right)^{2}}\left[\log\left(\frac{\xi''}{\xi'}\right)-\frac{\xi''-\xi'}{\xi''}\right](1).
\end{align*}
Since $\xi''(1)-\xi'(1)\neq0$, this is equivalent to
\begin{align*}
0 & \leq\left[\log\left(\frac{\xi''}{\xi'}\right)-\frac{\xi''-\xi'}{\xi''}-\xi\frac{\left(\xi''-\xi'\right)^{2}}{\xi'^{2}\xi''}\right](1)\\
 & =\left[\log\left(\frac{\xi''}{\xi'}\right)-\frac{\left(\xi''-\xi'\right)\left(\xi\left(\xi''-\xi'\right)+\xi'^{2}\right)}{\xi'^{2}\xi''}\right](1)=ABA(\nu',\nu''),
\end{align*}
where the last equality comes from the assumption that $\xi(1)=1$.
\end{proof}

\subsection{1RSB in 2+p models}

In this subsection we will prove \prettyref{thm:CL-ABA-2-p}. We begin
with the following observations. If we let $y=\frac{1}{c^{2}\xi'\left(1\right)}$,
we see that the master equation \prettyref{eq:master-equation} is
equivalent to 
\begin{equation}
\frac{\xi\left(1\right)}{\xi'\left(1\right)}=\frac{1}{y-1}\left(\frac{y}{y-1}\log y-1\right).\label{eq:root}
\end{equation}
We have that $0<\frac{\xi\left(1\right)}{\xi'\left(1\right)}\leq\frac{1}{2}$
with equality if and only if $\xi=\xi_{SK}$. Observe the following
fact which follows by a standard calculus exercise.
\begin{fact}
\label{fact:yxi}There exists a unique solution $y_{\xi}\in[1,\infty)$
to the equation \eqref{eq:root}. The solution satisfies $y_{\xi}=1$
if and only if $\xi=\xi_{SK}$ . Finally, if $y_{\xi}<1$, then \prettyref{eq:FP1RSB}
and \prettyref{eq:master-equation} have unique positive solutions.\end{fact}
\begin{proof}
Consider $a:(1,\infty)\to\R$
\[
a\left(y\right)=\frac{1}{y-1}\left(\frac{y}{y-1}\log y-1\right)
\]
It is then a calculus exercise to show that
\begin{itemize}
\item $a$ is continuous and strictly decreasing on $(1,\infty)$
\item $a\left(1+\right)=\frac{1}{2}$ and $a\left(+\infty\right)=0$
\end{itemize}
Therefore, $a$ is invertible, $a^{-1}:\left(0,1\right)\to\R_{+}$,
and $a^{-1}$ is continuous. This proves the first two claims. For
the third claim, note $y_{\xi}=1$ if and only if $\frac{\xi\left(1\right)}{\xi'\left(1\right)}=\frac{1}{2}$
which it true if and only if $\xi=\xi_{SK}$.
\end{proof}
With this in hand, we observe that one can rewrite the non-negative
replicon eigenvalue and pure-like or critical conditions as upper
and lower bounds on $y_{\xi}$ respectively. More precisely, 
\begin{lem}
\label{lem:ybounds}Let $\xi\neq\xi_{SK}$. Then $\xi$ satisfies
the ``non-negative replicon eigenvalue'' and ``pure-like or critical''
conditions if and only if 
\begin{equation}
\frac{\xi''\left(1\right)}{\xi'\left(1\right)}\leq y_{\xi}\leq\frac{\xi'\left(1\right)}{\xi''\left(0\right)}.\label{eq:y-bounds}
\end{equation}
\end{lem}
\begin{proof}
Let $m,c>0$ solve \eqref{eq:FP1RSB} and let $\eta$ be the formal
conjugate to $\phi=m(1-t)+c$. Then
\[
\eta''(t)=\frac{1}{(m(1-t)+c)^{2}}.
\]
By definition, the ``non-negative replicon eigenvalue'' condition
is that $\frac{1}{(m+c)^{2}}\geq\xi''(0)$. By \eqref{eq:FP1RSB}
this is equivalent to $c^{2}(\xi'(1))^{2}\geq\xi''(0)$. Similarly,
by \prettyref{lem:purelikeorcrit} we have that the ``pure-like or
critical'' condition is equivalent to $\frac{1}{c^{2}}\geq\xi''(1)$.
Now by \prettyref{lem:master-equation} and \prettyref{fact:yxi},
$y_{\xi}=\frac{1}{c^{2}\xi'(1)}$. The result immediately follows.
\end{proof}
For the remainder of this section we will be specifically addressing
models of the form $\xi_{\mu}\left(t\right)=\mu t^{2}+\left(1-\mu\right)t^{p}$
where $p\geq3$ is fixed and $0\leq\mu<1$. Then we refer to the point
$y_{\xi}$ as $y_{\mu}$. 

Define 
\begin{equation}
h_{\mu}=\frac{1}{\phi_{\mu}^{2}}-\xi_{\mu}'',\label{eq:h}
\end{equation}
where $\phi_{\mu}\left(t\right)=c_{\mu}+m_{\mu}\left(1-t\right),$
and where $c,m$ are determined from $\mu$ by solving \eqref{eq:FP1RSB}.
Observe the following lemma.
\begin{lem}
\label{lem:rootbd}(root bound) Let $\mu\in[0,1]$ be such that $y_{\mu}$
satisfies \eqref{eq:y-bounds}. Then the equation $h_{\mu}\left(t\right)=0$
has at most two solutions in $(0,1)$.\end{lem}
\begin{proof}
We suppress the subscript $\mu$ for readability when convenient.
Let $Z=\phi^{2}h=1-\phi^{2}\xi''$ and note that the roots of $Z$
and $h$ are the same. Observe that
\[
\phi^{2}\left(t\right)=\left(c+m\left(1-t\right)\right)^{2}=c^{2}\left(1+\left(\frac{1}{c^{2}\xi'\left(1\right)}-1\right)\left(1-t\right)\right)^{2}=C\left(\xi,y\right)\left(\frac{y}{y-1}-t\right)^{2}
\]
where $C\left(\xi,y\right)=\frac{\left(y-1\right)^{2}}{\xi'\left(1\right)y}.$
In particular, $Z_{\xi}\left(t\right)=1-C\left(\xi,y_{\xi}\right)\left(\frac{y_{\xi}}{y_{\xi}-1}-t\right)^{2}\xi''\left(t\right)$.
Specializing to the case of $2+p$, with $\xi=\xi_{\mu}$, we have
\begin{align*}
Z_{\mu}\left(t\right) & =1-C\left(\xi_{\mu},y_{\mu}\right)\left(\frac{y_{\mu}}{y_{\mu}-1}-t\right)^{2}\left(2\mu+p\left(p-1\right)\left(1-\mu\right)t^{p-2}\right)\\
C\left(\xi_{\mu},y\right) & =\frac{1}{2\mu+p\left(1-\mu\right)}\frac{\left(y-1\right)^{2}}{y}.
\end{align*}

Now the plan is to expand $Z$ and apply Descartes' rules of signs.
We begin by observing that by the non-negative replicon eigenvalue
condition (i.e., $h(0)\geq0$), we have that $Z(0)\geq0$, and similarly
by the pure-like or critical condition (i.e., that $h(1)\geq0$),
we have $Z(1)\geq0$. Let $x=\frac{y}{y-1}$, which is positive since
$y>1$, and expand $Z$ to find that
\begin{align*}
Z & =1-C\xi''\left(t\right)\left(x-t\right)^{2}=1-C\xi''\left(t\right)\left(x^{2}-2xt+t^{2}\right)\\
 & =Z\left(0\right)+C2\mu2xt-C2\mu t^{2}-Cp\left(p-1\right)\left(1-\mu\right)t^{p-2}x^{2}+Cp\left(p-1\right)\left(1-\mu\right)t^{p-1}2x-Cp\left(p-1\right)\left(1-\mu\right)t^{p}.
\end{align*}
Since $Z\left(0\right)\geq0$ we count at most three sign changes
(checking the cases $p=3$ and $p=4$ separately), so that by Descartes'
rule of signs there are at most three positive real roots. Since $Z\left(+\infty\right)=-\infty$
and $Z\left(1\right)\geq0$, one of those roots must be in $[1,\infty)$.
This proves the claim.
\end{proof}
Now we have a lemma regarding the sign of $Z$ near $0$. 
\begin{lem}
\label{lem:bdrybehavior}Let $y_{\mu}$ satisfy \eqref{eq:y-bounds}.
Then there is a $\delta>0$ such that $Z>0$ on $(0,\delta)$. \end{lem}
\begin{proof}
By \eqref{eq:y-bounds}, we see that $Z\left(0\right),Z\left(1\right)\geq0$.
Now for the first claim: If $Z\left(0\right)>0$ then the result follows
by continuity. If $Z\left(0\right)=0$, then since $C,x\neq0$, and
since $Z\left(0\right)=0$ then you must have that $\mu\neq0$. Then
$Z\left(t\right)=4Cx\mu t+o\left(t\right)$ for $t\to0$, so the result
follows.
\end{proof}

\subsubsection{Proof of \prettyref{thm:CL-ABA-2-p}}

Before we begin the proof, recall that we are considering models of
the form $\xi_{\mu}\left(t\right)=\mu t^{2}+\left(1-\mu\right)t^{p}$
where $p\geq3$ is fixed and $0\leq\mu<1$. Recall also that we refer
to $y_{\xi}$ as defined in \prettyref{fact:yxi} as $y_{\mu}$. We
now begin the proof.
\begin{proof}
The ``only if'' part of the theorem is implied by \prettyref{cor:consistency1RSB}.
Now we show the ``if'' part. 

To begin, let $m,c>0$ solve \eqref{eq:FP1RSB} and let $\eta$ be
the formal conjugate to $\phi(t)=m(1-t)+c$. Assume that the ``non-negative
replicon eigenvalue'' and the ``pure-like or critical'' conditions
hold, i.e., assume that $\eta''(0)\geq\xi''(0)$ and that $\eta''(1)\geq\xi''(1)$.
(See \prettyref{lem:purelikeorcrit}.) Our goal is to show that $\phi$
and $\eta$ are optimal. Recall that by \prettyref{lem:suffoptimality1RSB},
this will hold provided that the obstacle condition, $\eta\geq\xi$
for all $t\in\left[0,1\right]$, holds. We will prove this by studying
the difference of the second derivatives, $h_{\mu}$ from \eqref{eq:h}. 

We begin by observing that $h_{\mu}$ undergoes at most two sign changes
in $(0,1)$. To see this, observe that by \prettyref{lem:ybounds},
$y_{\mu}$ satisfies \eqref{eq:y-bounds} . The observation then follows
by \prettyref{lem:rootbd}. Furthermore, by \prettyref{lem:bdrybehavior}
we know that $h_{\mu}$ is positive on a neighborhood of $0$. The
obstacle condition will now follow by a case analysis. 

In the first case, $h_{\mu}$ does not change sign. Then, $h_{\mu}\geq0$
on $(0,1)$ and hence $\eta-\xi$ is convex there. Since $\eta(0)=\xi(0)$
and $\eta'(0)=\xi'(0)$, it follows that $\eta\geq\xi$ on $[0,1]$. 

In the second case, $h_{\mu}$ undergoes one sign change. Then, there
exists $\delta\in(0,1)$ such that $h_{\mu}>0$ on $(0,\delta)$ and
$h_{\mu}\leq0$ on $[\delta,1)$. Thus $\eta-\xi$ is convex on $(0,\delta)$
and by the same argument as above we conclude that $\eta\geq\xi$
on $[0,\delta]$. Continuing, we have that $\eta(\delta)\geq\xi(\delta)$
and $\eta(1)\geq\xi(1)$, and also that $\eta-\xi$ is concave on
$[\delta,1)$. It immediately follows that $\eta\geq\xi$ on $[\delta,1]$
as required.

In the final case, $h$ undergoes two sign changes. Then, there exist
$\delta_{1},\delta_{2}\in(0,1)$ with $\delta_{1}<\delta_{2}$, and
such that $h_{\mu}>0$ on $(0,\delta_{1})\cup(\delta_{2},1)$ and
$h\leq0$ on $[\delta_{1},\delta_{2}]$. Using the boundary data at
$0$ and $1$ and that $\eta-\xi$ is convex on $(0,\delta_{1})\cup(\delta_{2},1)$,
we conclude that $\eta\geq\xi$ on $[0,\delta_{1}]\cup[\delta_{2},1]$.
Continuing, we have that $\eta(\delta_{i})\geq\xi(\delta_{i})$ for
$i=1,2$ and that $\eta-\xi$ is concave on $[\delta_{1},\delta_{2}]$.
It follows that $\eta\geq\xi$ on $[\delta_{1},\delta_{2}]$ as required.
\end{proof}

\section{Appendix }

In this appendix, we collect some basic results that are used throughout
the paper. In \prettyref{sub:csfunc-reform}, we explain the relation
between the Crisanti-Sommers functional as defined in this paper and
those occurring previously in the literature. In \prettyref{sub:On-the-sets-C-and-A},
we present basic properties of the sets $\cA,\cC,$ and $\cK_{h,\xi}$.
In \prettyref{sub:Square-Roots-of}, we present basic properties of
$S$ and $D$. In \prettyref{sub:Basic-Theorems-in}, we prove basic
results regarding sequential $\Gamma$-convergence.

\subsection{Reformulation of the Crisanti-Sommers functional \label{sub:csfunc-reform}}

The Crisanti-Sommers variational problem is usually posed in a different
form. Let $\mu\in\Pr[0,1]$, and define $q_{*}=\sup\text{supp}\,\mu$.
The spherical Crisanti-Sommers functional is sometimes defined as
\[
\tilde{\cP}\left(\mu;\beta,h,\xi\right)=\begin{cases}
\frac{1}{2}\left(\int_{0}^{1}\beta^{2}\xi'\left(s\right)\mu\left[0,s\right]ds+h^{2}\hat{\mu}\left(0\right)+\int_{0}^{q_{*}}\frac{ds}{\hat{\mu}\left(s\right)}+\log\left(1-q_{*}\right)\right) & q_{*}<1\\
+\infty & q_{*}=1
\end{cases}.
\]
This functional is not lower semi-continuous in the weak-$*$ topology.
In the original work of Crisanti-Sommers, the functional is only defined
for $q_{*}<1$. That the minimization problem is unchanged by replacing
the functional above with $\cP$ from \eqref{eq:CSfunc} can be seen
by the following lemma. On the other hand, the functional $\cP$ is
lower semi-continuous. As the original functional was not defined
for $q_{*}=1$, we call this functional the Crisanti-Sommers functional
without ambiguity.
\begin{lem}
\label{lem:Pisanextension}For $q_{*}<1$, we have that 
\[
\tilde{\cP}_{\beta,h,\xi}(\mu)=\cP_{\beta,h,\xi}(\mu).
\]
\end{lem}
\begin{proof}
To see this, observe that
\begin{align*}
\int_{0}^{q_{*}}\frac{ds}{\int_{s}^{1}\mu\left[0,t\right]\,dt}+\log\left(1-q_{*}\right) & =\int_{0}^{q}\frac{ds}{\int_{s}^{1}\mu\left[0,t\right]\,dt}+\log\left(1-q\right)=\int_{0}^{q}\frac{ds}{\int_{s}^{1}\mu[0,t]dt}-\frac{1}{1-s}ds
\end{align*}
for all $q\in[q_{*},1)$, so that
\[
\int_{0}^{q_{*}}\frac{ds}{\int_{s}^{1}\mu\left[0,t\right]\,dt}+\log\left(1-q_{*}\right)=\int_{0}^{1}\left(\frac{1}{\hat{\mu}\left(s\right)}-\frac{1}{1-s}\right)\,ds.
\]
(This last integral is well-defined as the integrand is non-negative.)
Now if we recall that $\xi'(0)=0$ and use the fundamental theorem
and Fubini, we see that
\[
\int_{0}^{1}\xi'(s)\mu[0,s]ds=\int_{0\leq t\leq s\leq1}\xi''(t)\mu[0,s]ds=\int_{0}^{1}\xi''(t)\int_{t}^{1}\mu[0,s]ds.
\]
 Grouping then gives the result. \end{proof}
\begin{cor}
$\cP$ is lower semi-continuous and $\tilde{\cP}$ is not lower semi-continuous
in the weak-$*$ topology. \end{cor}
\begin{proof}
That $\cP$ is lower semi-continuous is an application of Fatou's
lemma. To see that $\tilde{\cP}$ is not lower semi-continuous, take
\[
\mu(dt)=\frac{1}{2\sqrt{1-t}}dt
\]
and let $\mu_{n}=\frac{1}{2\sqrt{1-t}}\indicator{t\leq1-\frac{1}{n}}dt+\left(\int_{1-\frac{1}{n}}^{1}\frac{1}{2\sqrt{1-t}}dt\right)\delta_{1-\frac{1}{n}}$.
Observe that $\mu_{n}\to\mu$ weakly. Furthermore, $\mu_{n}[0,t]\downarrow\mu[0,t]$
so that by the monotone convergence theorem, $\cP(\mu_{n})\to\cP(\mu)$.
Thus by \prettyref{lem:Pisanextension}, $\tilde{\cP}(\mu_{n})\to\cP(\mu)$.
Now we claim that $\cP(\mu)<\infty$. Indeed, $\hat{\mu}(s)=\sqrt{1-s}$
so that
\[
0\leq\frac{1}{\hat{\mu}(s)}-\frac{1}{1-s}\leq\frac{1}{\sqrt{1-s}}\in L^{1}.
\]
 Thus $\liminf\tilde{\cP}(\mu_{n})<\infty$ while $\tilde{\cP}(\mu)=\infty$.
\end{proof}

\subsection{On the sets $\cA$, $\cC$, and $\cK_{h,\xi}$.\label{sub:On-the-sets-C-and-A}}

In this section, we state some important and basic properties of the
sets $\cA$, $\cC$, and $\cK_{h,\xi}$.

\subsubsection{On the set $\cA$}
\begin{lem}
\label{lem:A-set-cpt} Let $f\in L^{1}([0,1])$ be non-negative and
positive on a set of positive Lebesgue measure. Then for every $C>0$,
the set 
\[
\left\{ \nu\in\cA:\int f(s)\nu[s,1]ds\leq C\right\} 
\]
is weak-$*$ sequentially compact. In particular, on this set we have
that
\[
\nu\left([0,1]\right)\leq\frac{C}{\int_{0}^{1}\int_{0}^{t}f(s)dsdt}.
\]
\end{lem}
\begin{proof}
Write $d\nu=m(t)dt+c\delta_{1}$, then by Fubini we have that
\[
\int_{0}^{1}f(s)\nu[s,1]\,ds=\int_{0\leq s\leq t\leq1}f(s)d\nu(t)ds=\int_{0}^{1}\int_{0}^{t}f(s)dsd\nu(t)=\int_{0}^{1}F(t)m(t)dt+cF(1)
\]
where $F(t)=\int_{0}^{t}f(s)ds$. By the Harris-FKG inequality,
\[
\int_{0}^{1}F(t)m(t)dt\geq\int_{0}^{1}F(t)dt\int_{0}^{1}m(t)dt
\]
so that 
\[
\int_{0}^{1}f(s)\nu[s,1]\,ds\geq\int_{0}^{1}F(t)dt\int_{0}^{1}m(t)dt+cF(1)\geq\int_{0}^{1}F(t)dt\cdot\nu([0,1]).
\]

The sequential compactness result now follows from the mass bound
above and the fact that $\cA$ is weak-$*$ sequentially closed in
$\cM$.\end{proof}
\begin{lem}
\label{lem:A-C-bij} The sets $\cC$ and $\cA$ are in bijective correspondence.
This correspondence is given by $\phi(t)=\nu[t,1]$.\end{lem}
\begin{proof}
Define the map $\Psi:\phi(t)\mapsto-\phi'(s)ds+\phi(1)\delta_{1}.$
Since $\phi\in\cC$ we know that $\phi'\in L^{1}$ and has a cadlag,
monotone decreasing version. Furthermore, $-\phi'\geq0$. Thus $\Psi:\cC\to\cA$.
Then given $\phi\in\cC$ we have that 
\[
\phi(t)=\phi(1)-\int_{t}^{1}\phi'(t)dt=\int_{t}^{1}d\nu
\]
where $\nu=\phi'(s)ds+\phi(1)\delta_{1}$, so that $\Psi$ is injective.
Now let $\nu\in\cA$ and let $\phi(t)=\int_{t}^{1}d\nu$. By definition
of $\nu$, $\phi\in C\left([0,1]\right)$ and is non-increasing. Since
$m(t)=\phi'(t)$ is cadlag and monotone, $\phi$ is convex. Thus $\phi\in\cC$
and $\Psi(\phi)=\nu$ so that the map $\Psi$ is surjective.
\end{proof}

\subsubsection{On the set $\cC$ }
\begin{lem}
\label{lem:helly-thm-in-C} For every $C>0$ , the set $E=\cC\cap\left\{ \phi:\norm{\phi}_{\infty}\leq C\right\} $
is norm compact.\end{lem}
\begin{proof}
This is an application of Helly's selection theorem, Dini's theorem
on the pointwise limit of continuous, monotone functions, and the
fact that if $\phi\in\cC$ then $\norm{\phi'}_{L^{1}}\leq\norm{\phi}_{L^{\infty}}$.\end{proof}
\begin{lem}
\label{lem:C-set-cpt} Let $f\in L^{1}([0,1])$ be non-negative and
positive on a set of positive Lebesgue measure. Then for every $C>0$
the set 
\[
S=\left\{ \phi\in\cC:\int f\phi\leq C\right\} 
\]
is norm compact. \end{lem}
\begin{proof}
Observe first that it is closed. Now we claim that $S$ admits a uniform
$L^{\infty}$ bound. Then the result will follow from \prettyref{lem:helly-thm-in-C}. 

To see the uniform $L^{\infty}$ bound, note that it is enough to
bound $\phi(0)$ by monotonicity properties of $\cC$. Now by concavity
of $\phi$ and non-negativity of $f$, we have that
\[
\int f\phi\geq\int f(t)((1-t)\phi(0)+t\phi(1))dt\geq\phi(0)\int f(t)(1-t)dt+\phi(1)\int f(t)tdt.
\]
Hence, 
\[
\norm{\phi}_{L^{\infty}}=\phi(0)\leq\frac{\int f\phi}{\int f(t)(1-t)dt}.
\]

\end{proof}
Given $\phi\in\cC$, it is immediate that $\phi\in W^{1,1}\subset BV$.
However, it may not be true that $\phi'\in BV$. Thus, we must be
careful with how we define the meaning of $\phi'$ at the boundary
points $0,1$. We call
\begin{align*}
\phi'(1) & =\lim_{t\to1^{-}}\frac{\phi(1)-\phi(t)}{1-t}\\
\phi'(0) & =\lim_{t\to0^{+}}\frac{\phi(t)-\phi(0)}{t}.
\end{align*}
Since $\phi$ is concave, both limits exist, though they may be infinite
\emph{a priori}. However, it is apparent from the definition of $\cC$
that $\phi'(0)\in\R$ and $\phi'(1)\in\R\cup\{-\infty\}$.

On the other hand, as a result of \prettyref{thm:(Modified-Alexandrov)},
we can identify $\phi''=-\mu$ as distributions, where $\mu$ is a
non-negative Radon measure on $(0,1)$. The next lemma relates the
mass of this measure to the boundary values of $\phi'$. 
\begin{lem}
\label{lem:phi'finite}Let $\phi\in\cC$ and let $\phi''=-\mu$ as
elements of $\cD'$, where $\mu$ is a non-negative Radon measure
on $(0,1)$. Then, 
\[
\mu((0,1))=\phi'(0)-\phi'(1).
\]
Hence, $\phi\in\cC\cap\{\phi'\in BV\}$ if and only if $\phi'(1)>-\infty$. 
\end{lem}
We note that if $\phi\in\cC\cap\{\phi'\in BV\}$, then $\phi'$ has
well-defined trace at $0,1$, and $(Tr\phi')(t)=\phi'(t)$ for $t=0,1$.

Finally, we have the following useful approximation result.
\begin{lem}
\label{lem:C-finite-deriv-approx} If $\phi\in\cC$ and $\phi'(1)=-\infty$,
then there is a sequence $\{\phi_{n}\}\subset\cC$ with $\phi_{n}\to\phi$
and $-\infty<\phi'_{n}(1)\leq-n$.
\end{lem}

\subsubsection{On the set $\cK_{h,\xi}$}

Since $\eta\in\cK_{h,\xi}$ is convex, we can define $\eta'$ at the
boundary points $0,1$ by 
\begin{align*}
\eta'(1) & =\lim_{t\to1^{-}}\frac{\eta(1)-\eta(t)}{1-t}\\
\eta'(0) & =\lim_{t\to0^{+}}\frac{\eta(t)-\eta(0)}{t}.
\end{align*}
It also follows from convexity that $\eta''=\mu$ as elements of $\cD'$,
where $\mu$ is a non-negative Radon measure on $(0,1)$ (see \prettyref{thm:(Modified-Alexandrov)}). 

The next lemma shows that in fact $\eta'\in BV$. Hence, $\eta'$
has well-defined trace at $0,1$ and $(Tr\eta')(t)=\eta'(t)$ for
$t=0,1$.
\begin{lem}
\label{lem:K-xi-bv} Let $\eta\in\cK_{h,\xi}$ and let $\eta''=\mu$
as elements of $\cD'$. Then $\eta'(1)\leq\xi'(1)$ and $\eta'(0)\geq\xi'(0)$.
Hence, $\eta'\in BV$ and 
\[
\mu((0,1))=\eta'(1)-\eta'(0)\leq\xi'(1)-\xi'(0).
\]

\end{lem}
The next result describes a useful construction.
\begin{lem}
\label{lem:meas-BVP-unique} Let $\sigma$ be a finite Radon measure
on $(0,1)$ and let $a,b\in\R$. Then the boundary value problem 
\[
\begin{cases}
\eta''=\sigma\\
\eta'(0)=a\\
\eta(1)=b
\end{cases}
\]
has a unique solution in the class $\eta\in C([0,1])\cap\{\eta'\in BV\}.$
\end{lem}

\subsubsection{An integration by parts lemma}
\begin{lem}
\label{lem:(Integration-by-parts)}(Integration by parts) Suppose
that $\eta,\phi\in C([0,1])\cap\left\{ \phi'\in BV\right\} .$ Then
\[
\left(\eta'',\phi\right)=\eta'\phi\vert_{0}^{1}-\eta\phi'\vert_{0}^{1}+\left(\eta,\phi''\right).
\]
\end{lem}
\begin{proof}
Applying the integration by parts theorem from \cite[Section 5.3]{EvansGariepy92}
along with a straightforward approximation argument, we conclude that
\begin{align*}
\int_{(0,1)}\phi d\eta'' & =-\int_{(0,1)}\eta'\phi'dx+\phi\eta'\vert_{0}^{1}\\
\int_{(0,1)}\eta d\phi'' & =-\int_{(0,1)}\eta'\phi'dx+\phi'\eta\vert_{0}^{1}.
\end{align*}
Subtracting these gives the desired result.
\end{proof}

\subsubsection{Alexandrov's theorem}

The following is a modification of a theorem of Alexandrov. 
\begin{thm}
\label{thm:(Modified-Alexandrov)}\cite{EvansGariepy92}(Alexandrov)
Let $f:(0,1)\to\R$ be convex, then $f''=\mu$ as elements of $\cD'$
where $\mu$ is a Radon measure on $(0,1)$.
\end{thm}

\subsection{Square Roots of Positive Measures\label{sub:Square-Roots-of}}
\begin{lem}
We have the equality
\begin{align*}
S\left(\mu\right) & =\inf_{\phi\geq0}\,\int\phi d\mu+\int\frac{1}{\phi}dx=\inf_{\substack{\phi\geq0\\
\phi\,u.s.c.\\
\phi\in L^{\infty}
}
}\,\int\phi d\mu+\int\frac{1}{\phi}dx.
\end{align*}
\end{lem}
\begin{proof}
Begin by noting that
\[
\inf_{\phi\geq0}\,\int\phi\mu_{ac}dx+\int\frac{1}{\phi}dx\geq2\int\sqrt{\mu_{ac}}
\]
by the arithmetic-geometric inequality so that
\begin{align*}
S\left(\mu\right) & \leq\inf_{\substack{\phi\geq0\\
\phi\,u.s.c.\\
\phi\in L^{\infty}
}
}\,\int\phi d\mu+\int\frac{1}{\phi}dx\leq\inf_{\substack{\phi\geq0\\
\phi\in C^{\infty}([0,1])
}
}\int\phi d\mu+\int\frac{1}{\phi}dx=:G(\mu)
\end{align*}
Write $\mu=fdx+d\nu$ with $\nu\perp\cL$ where $\cL$ denotes the
Lebesgue measure. Let $f_{\epsilon}$ be a smooth, non-negative $L^{1}$
approximation to $f$. Find a Hahn decomposition $I=A\sqcup B$ such
that $\nu\left(A\right)=0$, $\cL\left(B\right)=0$. Given any $U\supset B$
open, $m,\lambda>0$, define 
\[
\phi=\phi_{U,\epsilon,m,\lambda}=\begin{cases}
\frac{1}{\sqrt{f_{\epsilon}}+m} & U^{c}\\
\lambda\wedge\frac{1}{\sqrt{f_{\epsilon}}+m} & U
\end{cases}
\]
for $m>0$. Note that $\phi$ is upper semi-continuous, non-negative,
and satisfies the bound
\[
\phi\leq\frac{1}{m}\wedge\frac{1}{\sqrt{f_{\epsilon}}+m}.
\]
So, we get that
\[
G\left(\mu\right)\leq\int\phi d\mu+\int\frac{1}{\phi}dx.
\]

We have that
\begin{align*}
\int\phi d\mu & =\int\phi fdx+\int\phi d\nu=\int\phi\left(f-f_{\epsilon}\right)dx+\int\phi\left(f_{\epsilon}dx+d\nu\right)\leq\frac{1}{m}\norm{f-f_{\epsilon}}_{L^{1}}+\int\phi\left(f_{\epsilon}dx+d\nu\right).
\end{align*}
Also we have that 
\begin{align*}
\int\phi\left(f_{\epsilon}dx+d\nu\right) & =\int\phi f_{\epsilon}dx+\int_{B}\phi d\nu\leq\int\sqrt{f_{\epsilon}}\,dx+\lambda\nu(U)
\end{align*}
and that
\begin{align*}
\int\frac{1}{\phi}dx & =\int_{I\backslash U}\left(\sqrt{f_{\epsilon}}+m\right)\,dx+\int_{U}\frac{1}{\lambda}\vee\left(\sqrt{f_{\epsilon}}+m\right)\,dx\leq\int\sqrt{f^{\epsilon}}\,dx+m+\frac{1}{\lambda}\cL\left(U\right).
\end{align*}
Adding up, we get that
\[
G\left(\mu\right)\leq2\int\sqrt{f_{\epsilon}}\,dx+\lambda\nu\left(U\right)+\frac{1}{\lambda}\cL\left(U\right)+m+\frac{1}{m}\norm{f-f_{\epsilon}}_{L^{1}}
\]
and optimizing in $m$ and $\lambda$ gives
\[
G\left(\mu\right)\leq2\int\sqrt{f_{\epsilon}}\,dx+2\sqrt{\nu\left(U\right)\cL\left(U\right)}+2\norm{f-f_{\epsilon}}_{L^{1}}^{1/2}.
\]
Taking $U\downarrow B$ and then $f_{\epsilon}\to f$ in $L^{1}$
proves that$G\left(\mu\right)\leq2\int\sqrt{f}\,dx=2\int\sqrt{\mu_{ac}}.$

Therefore, we have shown that $S\leq G\leq2\int\sqrt{\mu_{ac}}\leq S$
as desired.\end{proof}
\begin{cor}
\label{cor:sqrt-usc}The functional $S:\cM_{+}\to\R$ is upper semi-continuous
in the weak-$*$ and norm topologies on $\cM_{+}$.\end{cor}
\begin{proof}
We proved above that
\[
S(\mu)=\inf_{\substack{\phi\geq0\\
\phi\,u.s.c.\\
\phi\in L^{\infty}
}
}\,\int\phi d\mu+\int\frac{1}{\phi}dx.
\]
Note that for a fixed such $\phi$, the functional $\mu\to\int\phi d\mu$
is upper semi-continuous in the norm and weak-$*$ topologies. Indeed,
the assumptions on $\phi$ imply that 
\[
\int\phi d\mu=\inf_{\substack{\psi\geq\phi\\
\psi\in C\left[0,1\right]
}
}\int\psi d\mu.
\]
Therefore it can be written as the point-wise infimum of continuous
functionals, hence it is upper semi-continuous. As a result $S(\mu)$
is the point-wise infimum of upper semi-continuous functions so that
it is upper semi-continuous.\end{proof}
\begin{cor}
\label{cor:D-usc}The functional $D:\cK_{h,\xi}\to\R$ is upper semi-continuous
in the norm topology. 
\end{cor}

\subsection{Basic Theorems in $\Gamma$-convergence.\label{sub:Basic-Theorems-in}}

In the following, let $X$ be a topological space and let $F_{\beta},F:X\to[-\infty,\infty]$.
Recall that the family of functions $F_{\beta}$ is said to be \emph{sequentially
equi-coercive} if for every $C>0$ there is a sequentially compact
set $K\subset X$ such that for every $\beta$, 
\[
\left\{ F_{\beta}\leq C\right\} \subset K.
\]

Furthermore recall that a sequence of functions $G_{\beta}\to G$
continuously if for every $x_{\beta}\to x$, $G_{\beta}(x_{\beta})\to G(x)$.
\begin{thm}
\label{thm:FTGC}(Fundamental Theorem of $\Gamma$-convergence) Suppose
F is not identically infinite. Suppose further that $F_{\beta}\stackrel{\Gamma}{\to}F$,
$F_{\beta}$ are sequentially equi-coercive, F has a unique minimizer,
$x_{\beta}$ are minimizers of $F_{\beta}$, and $x$ is the unique
minimizer of $F$. Then
\[
\lim_{\beta\to\infty}\min F_{\beta}=\min F
\]
and $x_{\beta}\to x.$\end{thm}
\begin{proof}
Without loss of generality, $F_{\beta},F$ are finite. Fix a sequence
$\beta_{n}\to\infty$. By the $\Gamma-\limsup$ inequality, we have
that there is a sequence $y_{n}\to x$ 
\[
F(x)\geq\limsup F_{\beta_{n}}(y_{n})\geq\limsup\inf_{z\in X}F_{\beta_{n}}(z)\geq\limsup F_{\beta_{n}}(x_{\beta_{n}}).
\]
As a result, there is a $C$ such that $F_{\beta_{n}}(x_{\beta_{n}})\leq C$.
By sequential equi-coercivity, this implies that $\{x_{\beta_{n}}\}\subset K$
for some sequentially compact set. In particular, there is a further
subsequence such that $x_{\beta_{n_{k}}}\to y$ for some $y\in X$.
Then 
\[
F(x)\geq\liminf F_{\beta_{n_{k}}}(x_{\beta_{n_{k}}})\geq F(y)
\]
 by the above inequalities and the $\Gamma-\liminf$ inequality. Then
since $x$ is the unique minimizer, $x=y$. Thus by the subsequence
principle, $x_{\beta_{n}}\to x$. 
\end{proof}
The proof of the following is self-evident.
\begin{thm}
\label{thm:stability-GC} (Stability under continuous perturbations)
Suppose that $F_{\beta}\stackrel{\Gamma}{\to}F$ and that $G_{\beta}\to G$
continuously. Then $F_{\beta}+G_{\beta}\stackrel{\Gamma}{\to}F+G$.
\end{thm}
\bibliographystyle{plain}
\bibliography{parisimeasures}

\end{document}